\documentclass[11pt]{amsart}
\usepackage{amsmath}
\usepackage{mathtools}

\usepackage{graphicx}
\usepackage[dvipsnames, table]{xcolor}
\definecolor{tiffanygreen}{RGB}{74, 148, 150}
\definecolor{mutedcyan}{RGB}{68,114,196}
\definecolor{othermutedcyan}{RGB}{70,103,150}
\usepackage[colorlinks=true, allcolors=tiffanygreen]{hyperref}
\usepackage{amsfonts}
\usepackage{amsthm}
\usepackage{newlfont}
\usepackage{amscd}
\usepackage{amsgen}
\usepackage{amssymb} 
\usepackage{mathrsfs}	
\usepackage{longtable}
\usepackage{listings}
\usepackage{extarrows}
\usepackage{tikz}
\usepackage{tikz-cd}
\usepackage{verbatim}
\numberwithin{equation}{section}
\usepackage[all]{xy}
\usepackage{color}
\usepackage{amssymb}
\usepackage[left=3cm,right=3cm]{geometry}
\usepackage{tikz}
\usepackage{tikz-cd}
\usepackage{mathtools}
\usepackage{bm} \usepackage{dynkin-diagrams}
\usetikzlibrary{matrix,shapes,arrows,decorations.pathmorphing}
\usepackage{calligra}
\usepackage{mathrsfs}
\usepackage{enumitem} \usepackage{tgpagella} \usepackage[parfill]{parskip} \usepackage{float}\restylefloat{table}
\usepackage{ytableau}
\usepackage{adjustbox}
\usetikzlibrary{calc,matrix,positioning}

\tikzset{ 
    table/.style={
        matrix of nodes,
        row sep=-\pgflinewidth,
        column sep=-\pgflinewidth,
        nodes={rectangle, text width=2.5em, text height = 1.5em, align=center},
        text depth=1.25ex,
        text height=2.5ex,
        nodes in empty cells
    },
}

\let\blb\mathbb
\def\AA{{\blb A}}

\def\CC{{\blb C}}

\def\LL{{\blb L}}

\def\PP{{\blb P}}

\def\RR{{\blb R}}

\def\VV{{\blb V}}

\def\ZZ{{\blb Z}}

\let\cal\mathcal
\def\Ac{{\cal A}}
\def\Bc{{\cal B}}
\def\Cc{{\cal C}}
\def\Dc{{\cal D}}
\def\Ec{{\cal E}}
\def\Fc{{\cal F}}
\def\Gc{{\cal G}}
\def\Hc{{\cal H}}
\def\Ic{{\cal I}}

\def\Kc{{\cal K}}
\def\Lc{{\cal L}}

\def\Nc{{\cal N}}
\def\Oc{{\cal O}}
\def\Pc{{\cal P}}

\def\Sc{{\cal S}}
\def\Tc{{\cal T}}
\def\Uc{{\cal U}}
\def\Vc{{\cal V}}
\def\Wc{{\cal W}}

\newcommand{\sfB}{\mathsf{B}}
\newcommand{\sfC}{\mathsf{C}}
\newcommand{\sfD}{\mathsf{D}}
\newcommand{\sfE}{\mathsf{E}}
\newcommand{\sfF}{\mathsf{F}}
\newcommand{\sfG}{\mathsf{G}}
\newcommand{\sfH}{\mathsf{H}}

\newcommand{\sfJ}{\mathsf{J}}

\newcommand{\sfL}{\mathsf{L}}
\newcommand{\sfM}{\mathsf{M}}

\newcommand{\sfP}{\mathsf{P}}
\newcommand{\sfQ}{\mathsf{Q}}
\newcommand{\sfR}{\mathsf{R}}
\newcommand{\sfS}{\mathsf{S}}
\newcommand{\sfT}{\mathsf{T}}
\newcommand{\sfU}{\mathsf{U}}
\newcommand{\sfV}{\mathsf{V}}
\newcommand{\sfW}{\mathsf{W}}
\newcommand{\sfX}{\mathsf{X}}
\newcommand{\sfY}{\mathsf{Y}}
\newcommand{\sfZ}{\mathsf{Z}}

\newtheorem{lemma}{Lemma}[section]
\newtheorem{proposition}[lemma]{Proposition}
\newtheorem{theorem}[lemma]{Theorem}
\newtheorem{introtheorem}{Theorem}

\newtheorem{corollary}[lemma]{Corollary}

\newtheorem{notation}[lemma]{Notation}
\newtheorem{assumption}[lemma]{Assumption}
\theoremstyle{remark}

\newtheorem{remark}[lemma]{Remark}

\def\dbcoh{D^b}

\def\Perf{\operatorname{Perf}}

\def\wt{\widetilde}

\def\arw{\longrightarrow}
\newcommand{\birarrow}{\mathrel{\tikz[baseline=-0.5ex]{\draw[<->,dashed] (0,0) -- node[above,inner sep=1pt] {$\scriptstyle\mathrm{bir}$} (1.1,0);}}}

\def\deg{\operatorname{deg}}

\DeclareMathOperator{\Gr}{Gr}

\DeclareMathOperator{\Bl}{Bl}
\DeclareMathOperator{\Sym}{Sym}

\DeclareMathOperator{\sHom}{\mathscr{H}\text{\kern -3pt {\calligra\large om}}\,}

\newcommand\quotient[2]{
        \mathchoice
            {\text{\raise1ex\hbox{$#1$}\Big/\lower1ex\hbox{$#2$}}}
            {#1\,/\,#2
            }
            {#1\,/\,#2
            }
            {#1\,/\,#2
            }
    }

\makeatletter
\def\namedlabel#1#2{\begingroup
    #2\def\@currentlabel{#2}\phantomsection\label{#1}\endgroup
}
\makeatother

\definecolor{codebg}{RGB}{248, 248, 248}
\definecolor{codeframe}{RGB}{220, 220, 220}

\usepackage{microtype}
\usepackage{booktabs}
\title[Categorical resolutions and birational geometry of nodal GM varieties]{Categorical resolutions and birational geometry of nodal Gushel--Mukai varieties}

\author[K. Grzelakowski]{Kacper Grzelakowski}
\address{K. Grzelakowski: Department of Mathematics and Computer Science \\ University of Lodz \\ S. Banacha 22 \\ 90-238 Lodz, Poland.}
\email{kacper.grzelakowski@gmail.com}

\author[M. Rampazzo]{Marco Rampazzo}
\address{M. Rampazzo: Yau Mathematical Sciences Center \\ Tsinghua University \\ Shuangqing Complex Bldg.\\ Haidian district, Beijing, China.}
\email{marcorampazzo@mail.tsinghua.edu.cn, marco.rampazzo.90@icloud.com}

\author[S. Zhang]{Shizhuo Zhang}
\address{S. Zhang: School of Mathematics, Sun Yat-sen University, Guangzhou 510275, China}
\email{zhangshzh28@mail.sysu.edu.cn}

\begin{document}\begin{abstract}
We study ordinary Gushel--Mukai varieties with one node in dimensions three, four, and five, under some generality hypotheses. The blowup of the node is related by a relative Atiyah flop to a quadric fibration over $\PP^2$. We identify the categorical resolution of the Kuznetsov component with the Clifford component of this fibration. Then, we study hyperbolic equivalence, the relation with Verra varieties, and birational Torelli for general nodal threefolds. For fourfolds, we prove that a countable union of divisorial families is dense in the complex topology, with rational members and untwisted degree-two $K3$ categories. Finally, we construct a spinor modification of the associated conic bundle and compute its trace-free endomorphism bundle, which is not a sum of line bundles.
\end{abstract}
\maketitle
\enlargethispage{3pt}
\setcounter{tocdepth}{1}
\tableofcontents
\section{Introduction}

For $2\leq n\leq 5$, a smooth, ordinary Gushel--Mukai (GM) variety of dimension $n$ is defined as a transverse intersection:
\begin{equation*}
\sfX = \sfG(2, 5)\cap \sfQ\cap \PP^{n+4}
\end{equation*}
where $\sfQ$ is a smooth quadric hypersurface in $\PP^9$. Gushel classified smooth prime Fano threefolds of degree~10 \cite{gushel_fano_genus6}. Mukai extended the classification to smooth Fano varieties of coindex~3, degree~10, and Picard number~1 \cite{mukai1989biregular, mukai1993curves, mukai1995new}, assuming the existence of smooth members of the fundamental linear system, which was subsequently proved by Mella \cite{mella1999existence}. Such varieties have dimension $3\leq n\leq6$ and are either ordinary GM varieties as above, or double covers of linear sections of $\sfG(2,5)$ branched along ordinary GM varieties, called \emph{special} GM varieties; see also \cite[Theorem 2.16]{debarre_kuznetsov_classification}. A GM variety has an associated Eisenbud--Popescu--Walter (EPW) sextic. Its double cover, under generality assumptions, is a smooth hyperk\"ahler fourfold \cite{ogrady_epw}. The pair consisting of a GM variety and its associated double EPW sextic exhibits many parallels with the relationship between a cubic fourfold and its Fano variety of lines (which is again hyperk\"ahler \cite{beauville_donagi}). Such pairs relate Fano varieties of dimensions three to six to hyperk\"ahler fourfolds. See \cite{Logachev2012FanoGenus6, IlievManivel2011FanoDegreeTen, Nagel1998GeneralizedHodge, DebarreKuznetsov2019GushelMukai} for the Hodge diamonds below: {\renewcommand{\arraystretch}{0.5} \setlength{\arraycolsep}{3pt}{
\begin{equation*}
\begin{array}{c|c|c|c}
n = 2 & n = 3 & n = 4 & n = 5 \\
\hline &&&\\
\begin{array}{ccccc}
&& 1 && \\
& 0 && 0 & \\
1 && 20 && 1 \\
& 0 && 0 & \\
&& 1 && \\
\end{array} &
\begin{array}{ccccccc}
&&& 1 &&& \\
&& 0 && 0 && \\
& 0 && 1 && 0 & \\
0 && 10 && 10 && 0 \\
& 0 && 1 && 0 & \\
&& 0 && 0 && \\
&&& 1 &&&
\end{array} &
\begin{array}{ccccccccc}
&&&& 1 &&&& \\
&&& 0 && 0 &&& \\
&& 0 && 1 && 0 && \\
& 0 && 0 && 0 && 0 & \\
0 && 1 && 22 && 1 && 0 \\
& 0 && 0 && 0 && 0 & \\
&& 0 && 1 && 0 && \\
&&& 0 && 0 &&& \\
&&&& 1 &&&&
\end{array} &
\begin{array}{ccccccccccc}
&&&&& 1 &&&&& \\
&&&& 0 && 0 &&&& \\
&&& 0 && 1 && 0 &&& \\
&& 0 && 0 && 0 && 0 && \\
& 0 && 0 && 2 && 0 && 0 & \\
0 && 0 && 10 && 10 && 0 && 0 \\
& 0 && 0 && 2 && 0 && 0 & \\
&& 0 && 0 && 0 && 0 && \\
&&& 0 && 1 && 0 &&& \\
&&&& 0 && 0 &&&& \\
&&&&& 1 &&&&& \\
\end{array}
\end{array}
\end{equation*}
} }

In dimension four, the middle Hodge structure of a GM variety has $h^{3,1}=h^{1,3}=1$ and $h^{4,0}=h^{0,4}=0$, so the variety is a \emph{Fano variety of $K3$ type} \cite[Proposition 3.1]{DebarreKuznetsov2019GushelMukai}.

The notion of a Fano variety of $K3$ type is expected to have a categorical counterpart at the level of derived categories in terms of a semiorthogonal decomposition containing a $K3$ category (i.e. its Serre functor is a shift by two and its Hochschild cohomology is isomorphic, as a graded vector space, to that of a $K3$ surface). Indeed, in both the cubic fourfold and the GM fourfold cases, the bounded derived category of coherent sheaves admits a semiorthogonal decomposition whose nontrivial component is a $K3$ category. More precisely, for a smooth cubic fourfold $\sfY \subset \mathbb{P}^5$, one has a semiorthogonal decomposition
\begin{equation*}
\dbcoh(\sfY) = \langle \mathcal{K}u(\sfY),\, \mathcal{O}_{\sfY},\, \mathcal{O}_{\sfY}(1),\, \mathcal{O}_{\sfY}(2) \rangle,
\end{equation*}
where the Kuznetsov component $\mathcal{K}u(\sfY)$ is a  $K3$ category \cite{kuznetsov_cubic}. Similarly, for a smooth ordinary GM fourfold $\sfX$, there is a semiorthogonal decomposition
\begin{equation*}
\dbcoh(\sfX) = \langle \mathcal{K}u(\sfX),\, \mathcal{O}_{\sfX},\, \mathcal{U}_{\sfX}^\vee,\, \mathcal{O}_{\sfX}(1),\, \mathcal{U}_{\sfX}^\vee(1) \rangle,
\end{equation*}
where $\mathcal{U}_{\sfX}$ denotes the restriction of the tautological bundle from the Grassmannian, and the Kuznetsov component $\mathcal{K}u(\sfX)$ is again a $K3$ category \cite{kuznetsovperry}. In odd dimension, the Kuznetsov component of a smooth GM variety is instead of Enriques type: its Serre functor is the composition of the shift by two with a nontrivial involution \cite[Proposition 2.6(2)]{kuznetsovperry}.

For cubic and GM fourfolds, moduli spaces of Bridgeland-stable objects in the Kuznetsov components have been constructed and shown to be hyperk\"ahler varieties, providing Fano--hyperk\"ahler pairs in a completely categorical way \cite{LahozLehnMacriStellari2018, BayerLahozMacriNuerPerryStellari2021, perry_pertusi_zhao, guo2024conics}.

For a very general smooth cubic fourfold, the Kuznetsov component is not equivalent to the derived category of a projective $K3$ surface \cite[Theorem 1.1]{addington_thomas_cubic_fourfolds}. For a cubic fourfold with one node, the corresponding categorical resolution is geometric \cite{kuznetsov_cubic}. Inspired by \cite{kuznetsov_cubic, cattani_et_al}, we give a description of a ``natural'' Kuznetsov component of the blowup of a one-nodal, ordinary GM fourfold in its node, and we describe a similar picture for threefolds and fivefolds.

Let us start by briefly reviewing what is known for cubic fourfolds.
\subsection*{Derived categories of nodal cubic fourfolds}
Consider a cubic hypersurface $\sfY\subset\PP^5$ with a single node $y\in \sfY$, and its blowup $\pi: \wt \sfY\arw \sfY$ at the node. Call $\sfE$ the exceptional divisor, and $\sfH$ the pullback of the hyperplane class of $\PP^5$. Then, by applying Kuznetsov's result on categorical resolutions of singularities \cite[Theorem 1]{kuznetsov_resolutions_of_singularities}, there exists an admissible subcategory $\wt \Dc\subset\dbcoh\bigl(\wt \sfY\bigr)$ such that one has a semiorthogonal decomposition
\begin{equation}\label{eq:cubic_fourfold_categorical_resolution}
\dbcoh\bigl(\wt \sfY\bigr) = \langle \Oc_{\sfE}(2\sfE), \Oc_{\sfE}(\sfE), \wt \Dc\rangle,
\end{equation}
and such that the functors $\mathbf{R}\pi_*: \wt\Dc\arw \dbcoh(\sfY)$ and $\mathbf{L}\pi^*: \Perf(\sfY)\arw \wt\Dc$ are an adjoint pair (respectively right and left adjoint). However, this decomposition can be refined to the following \cite{kuznetsov_cubic}:

\begin{equation}\label{eq:cubic_resolution_lefschetz_decomposition}
\dbcoh\bigl(\wt \sfY\bigr) = \langle \Oc_{\sfE}(2\sfE), \Oc_{\sfE}(\sfE), \wt{\Kc u}(\sfY), \Oc, \Oc(\sfH), \Oc(2\sfH)\rangle.
\end{equation}

In light of Equation~\eqref{eq:cubic_fourfold_categorical_resolution}, $\wt{\Kc u}(\sfY)$ can be interpreted as the \emph{categorical resolution of the Kuznetsov component}. One can prove that $\wt{\Kc u}(\sfY)$ is equivalent to the derived category of a $K3$ surface $\sfZ\subset\PP^4$ with the following argument \cite{kuznetsov_cubic}:
\begin{enumerate}
\item First, $\wt \sfY$ is a divisor in the blowup $\wt\PP^5 := \operatorname{Bl}_y\PP^5$, where the latter has a second contraction $q:\wt\PP^5\arw \PP^4$ which is the projectivization of a vector bundle of rank two.
\item Call $h$ the pullback of the hyperplane class from $\PP^4$: then, as divisor classes in $\wt \PP^5$, one sees that $[\wt \sfY] = 3\sfH-2\sfE = 2h+\sfH$, i.e. the induced morphism $\wt \sfY\arw \PP^4$ is a blowup of $\PP^4$ in the $K3$ surface $\sfZ$ given by the vanishing of a section of $q_*\Oc(2h+\sfH)\simeq \Oc(2h)\oplus\Oc(3h)$. Let us call $\sfD$ the exceptional divisor, $\bar q:= q|_{\sfD}$, and $j$ the embedding of $\sfD$ in $\wt \sfY$.
\item By Orlov's blowup formula \cite{orlovblowup}, one has a semiorthogonal decomposition:
\begin{equation}\label{eq:cubic_resolution_blowup_decomposition}
\dbcoh\bigl(\wt \sfY\bigr) = \langle j_*\bar q^*\dbcoh(\sfZ)\otimes\Oc(\sfD), \Oc, \Oc(h), \Oc(2h), \Oc(3h), \Oc(4h) \rangle.
\end{equation}

Finally, comparing the decompositions in Equation~\eqref{eq:cubic_resolution_lefschetz_decomposition} and Equation~\eqref{eq:cubic_resolution_blowup_decomposition}, one can prove that $\wt{\Kc u}(\sfY) \simeq \dbcoh(\sfZ)$ by a sequence of mutations of the exceptional objects generating their semiorthogonal complements.
\end{enumerate}

\subsection*{Main result}
Let $\sfX$ be an ordinary nodal GM variety of dimension $n$, with $3\leq n\leq5$, in the open locus described in Assumption~\ref{ass:generality}, and let $\varphi:\wt \sfX\arw \sfX$ be the blowup of its node $x$. Write $\sfE$ for the exceptional divisor, $\sfH$ for the pullback of the Pl\"ucker class, and $\Uc$ for the restricted tautological bundle and its pullback. We denote by $\Kc u(\sfX)$ the right orthogonal to $\langle\Oc,\Uc^\vee,\ldots,\Oc((n-3)\sfH),\Uc^\vee((n-3)\sfH)\rangle$ in $\dbcoh(\sfX)$. The categorical resolution $\wt{\Kc u}(\sfX)$ of $\Kc u(\sfX)$ is specified by the following semiorthogonal decomposition (see Proposition~\ref{prop:resolved_gm_semiorthogonal_decomposition}):
\begin{equation}
\begin{split}\label{eq:intro_gm_categorical_resolution}
\dbcoh\bigl(\wt \sfX\bigr)=\langle&\Oc_{\sfE}((n-2)\sfE),\ldots,\Oc_{\sfE}(\sfE), \wt{\Kc u}(\sfX),\\
&\Oc,\Uc^\vee,\ldots,\Oc\bigl((n-3)\sfH\bigr),\Uc^\vee\bigl((n-3)\sfH\bigr)\rangle.
\end{split}
\end{equation}
The first goal is to give a geometric interpretation of its residual category. A natural attempt would be to adapt Kuznetsov's nodal cubic construction, i.e. comparing the two semiorthogonal decompositions coming from the two extremal contractions of the resolution of singularities. However, in the GM case, the second contraction of $\wt \sfX$ is a small contraction to a singular variety $\sfT$, rather than the blowup of a smooth center. For this reason, a direct application of the cubic fourfold approach seems not to be a viable choice. Nevertheless, we can find another decomposition by moving to a different birational model. Indeed, a second small resolution of $\sfT$, obtained by restricting a construction of Kuznetsov and Prokhorov \cite{kuznetsov_prokhorov}, gives a quadric fibration $\sfR$ over $\PP^2$. The birational map is a relative Atiyah flop. Writing $p:\sfY\arw\wt\sfX$ and $q:\sfY\arw\sfR$ for the maps from its common blowup, it provides an equivalence of categories
\begin{equation*}
\Phi=\mathbf{R}q_*\mathbf{L}p^*:\dbcoh\bigl(\wt \sfX\bigr)\xrightarrow{\sim}\dbcoh(\sfR).
\end{equation*}
Then, we can compare Equation~\eqref{eq:intro_gm_categorical_resolution} with Kuznetsov's semiorthogonal decomposition for quadric fibrations \cite[Theorem 4.2]{kuznetsov_quadric_fibrations_intersection_quadrics}. The comparison is carried out by mutations of exceptional objects. We obtain the following result. For a matter of exposition, we refer to Assumption~\ref{ass:generality} for the appropriate generality assumptions on $\sfX$, without stating them explicitly here.

\begin{introtheorem}[Theorem~\ref{thm:kuznetsov_clifford_equivalence}]\label{thm:intro_kuznetsov_clifford_equivalence}
Let $\sfX$ be an $n$-dimensional ordinary GM variety, with $3\leq n\leq5$, satisfying Assumption~\ref{ass:generality}. There is an equivalence
\begin{equation*}
\wt{\Kc u}(\sfX)\simeq\dbcoh(\PP^2,\Cc),
\end{equation*}
where $\Cc$ is the even Clifford algebra of a quadric fibration $\sfR\arw\PP^2$ of relative dimension $n-2$. In particular,
\begin{equation*}
\wt{\Kc u}(\sfX_4)\simeq\dbcoh(\sfS,\alpha),
\end{equation*}
where $\sfS\arw\PP^2$ is a smooth double cover branched along a sextic, hence a degree-two polarized $K3$ surface, and $\alpha\in\operatorname{Br}(\sfS)$ is the associated Brauer class.
\end{introtheorem}

\subsection*{Applications}
We discuss some applications of Theorem~\ref{thm:intro_kuznetsov_clifford_equivalence} to smooth Verra varieties and nodal Gushel--Mukai varieties.
\subsubsection*{The relationship with Verra threefolds and fourfolds}
A Verra threefold is a smooth divisor of bidegree $(2,2)$ in $\PP^2\times\PP^2$ \cite[Section 1.1.6]{iliev1997theta}, while a Verra fourfold is a double cover of this product branched along such a divisor \cite[Section 0.3]{iliev_kapustka_kapustka_ranestad_verra}. Each projection defines a quadric fibration, with a Clifford component which we denote by $\Kc u(\sfV)$ once a projection has been chosen.

\begin{introtheorem}[Theorem~\ref{thm:gm_verra_hyperbolic_comparison} and Corollary~\ref{cor:gm_verra_clifford_equivalence}]\label{thm:intro_gm_verra_comparison}
For a general one-nodal ordinary GM variety $\sfX$ of dimension $n=3$ or $4$, there is a Verra variety $\sfV$ such that
\begin{equation*}
\sfX\birarrow \sfV,\qquad \wt{\Kc u}(\sfX)\simeq\Kc u(\sfV).
\end{equation*}
\end{introtheorem}

The birational maps of Theorem~\ref{thm:intro_gm_verra_comparison} are obtained by composing the flop $\wt\sfX\dashrightarrow \sfR$ with a hyperbolic equivalence over $\PP^2$ (see \cite[Proposition 1.1]{kuznetsov_hyperbolic_equivalence} and Section~\ref{subsec:hyperbolic_reduction} for more details about hyperbolic reductions and quadric fibrations). The birational correspondence with Verra threefolds was established by Debarre--Iliev--Manivel \cite[Theorem 4.5]{debarre_iliev_manivel_GM_Verra}, and the fourfold correspondence by Bini--Kapustka--Kapustka \cite[Corollary 5.2]{bini_kapustka2}. We give a uniform construction by hyperbolic reduction and identify the resulting Clifford category with the categorical resolution considered here. In particular, the description by means of hyperbolic reduction makes the equivalence of Kuznetsov components appear as a natural consequence\footnote{It is now well-known that hyperbolic equivalence induces an equivalence of Kuznetsov components of quadric fibrations \cite[Proposition 1.1(3)]{kuznetsov_hyperbolic_equivalence}; see also \cite[Corollary 1.8.8]{auel_bernardara_bolognesi_quadric_fibrations}.}.

\subsubsection*{Duality of GM varieties of different dimensions}
The smooth categorical dualities of Kuznetsov and Perry \cite[Theorem 1.6]{kuznetsov_perry_categorical_cones} suggest analogous questions for categorical resolutions of nodal varieties. We obtain the following statement.

\begin{introtheorem}[Theorem~\ref{thm:odd_gm_duality}]\label{thm:intro_gm_duality}
For a general one-nodal ordinary GM fivefold $\sfX_5$, there exists a one-nodal ordinary GM threefold $\sfX_3$ such that
\begin{equation*}
\wt{\Kc u}(\sfX_5)\simeq\wt{\Kc u}(\sfX_3).
\end{equation*}
\end{introtheorem}

\subsubsection*{Divisorial families of rational nodal fourfolds}
\begin{introtheorem}[Theorem~\ref{thm:divisorial_rational_fourfolds}]\label{thm:intro_divisorial_rational_fourfolds}
Fix a general hyperplane through $x$. In the open parameter space of quadratic equations satisfying Assumption~\ref{ass:generality}, there is a countable collection of algebraic divisors whose union is dense in the complex topology. Every member of these divisors defines a rational one-nodal ordinary GM fourfold $\sfX_4$ and an equivalence
\begin{equation*}
\wt{\Kc u}(\sfX_4)\simeq\dbcoh(\sfS),
\end{equation*}
where $\sfS$ is a degree-two polarized $K3$ surface.
\end{introtheorem}

For the divisorial families, the vanishing of the Brauer class implies rationality. This supports the expected relationship between rationality and geometric categorical resolutions.

\subsubsection*{Categorical Torelli theorem for nodal Gushel--Mukai threefolds}
For smooth general GM threefolds, the Kuznetsov component determines the birational class \cite{jacovskis2021categorical}. In the nodal case we use the larger smooth category supplied by the categorical resolution.

\begin{introtheorem}[Theorem~\ref{thm:categorical_torelli}]\label{thm:intro_categorical_torelli}
Let $\sfX_3$ and $\sfX_3'$ be general one-nodal ordinary GM threefolds. If
\begin{equation*}
\wt{\Kc u}(\sfX_3)\simeq\wt{\Kc u}(\sfX_3')
\end{equation*}
as $\CC$-linear triangulated categories, then $\sfX_3\sim \sfX_3'$.
\end{introtheorem}

The proof passes through the polarized intermediate Jacobians and the Torelli theorem for general Verra threefolds.

Finally we investigate spinor modifications, a technique introduced by Kuznetsov \cite[Theorem 1.2 and Lemma 2.21]{sasha_spinor_modifications} to construct birational maps between conic bundles while preserving their even Clifford algebras up to Morita equivalence, and hence their Clifford categories up to equivalence. The construction uses so-called \emph{abstract spinor bundles} and gives a more direct approach than a sequence of hyperbolic extensions and reductions. However, constructing such bundles is nontrivial in general, and we hope to shed some light on this problem by providing additional examples. In our setting, a flopping curve on $\sfR_3$ determines an abstract spinor bundle and hence another conic bundle over $\PP^2$, birational to $\sfR_3$ and with an equivalent Clifford category. The same kind of construction on a Verra threefold produces a canonical spinor bundle, up to a line-bundle twist from the base, and its modification is isomorphic to the original conic bundle.

\subsection*{Acknowledgments}
The authors would like to express their gratitude to Grzegorz and Micha\l\ Kapustka for introducing them to this problem, and to Arend Bayer, Alexander Kuznetsov, and Laurent Manivel for their valuable comments on the first draft. Moreover, MR would like to thank Ignacio Barros, Soheyla Feyzbakhsh, Giovanni Mongardi, and Ying Xie for helpful and inspiring suggestions. SZ thanks Zhiyu Liu, Xun Lin, Fei Peng, and Zheng Zhang for their useful discussions. Part of the work was carried out when the authors visited IBS-CGP, IMS-Shanghai Tech University, the Jagiellonian University, IMPAN, and Soochow University. We are grateful for their excellent hospitality and support. KG was supported by 2024/53/B/ST1/00161. MR was supported by Fonds voor Wetenschappelijk Onderzoek -- Vlaanderen (FWO, Research Foundation -- Flanders) -- G0D9323N, by the European Research Council (ERC) grant agreement No. 817762, and by the Polish National Science Center project number 2024/55/B/ST1/02409. SZ is supported by SYSU Starting Grant No. 34000-12256019. The research was supported by the program Excellence Initiative at the Jagiellonian University in Krakow (ID.UJ).

\subsection*{Notations \& conventions}
We shall work over the field of complex numbers. Given $m$ equations $f_1, \dots, f_m$, with the notation $\VV(f_1, \dots, f_m)$ we denote the intersection of their vanishing loci. We adopt the same notation when the vector $(f_1, \dots, f_m)$ is a section of a rank $m$ vector bundle. We use the symbol ``$\simeq$'' for isomorphisms, and ``$\sim$'' for birational equivalences. By $\dbcoh(\sfX)$, we denote the bounded derived category of coherent sheaves of a variety $\sfX$, and by $\Perf(\sfX)$ its full subcategory of perfect complexes. We omit dimension subscripts when the dimension is fixed, retaining them when varieties of different dimensions are compared.

All equivalences are $\CC$-linear and exact. Projectivizations of vector spaces and vector bundles are denoted by $\PP$ and parametrize lines; in particular, for a vector bundle $\Fc$ and the projection $\eta:\PP(\Fc)\arw \sfB$, the choice of a relative hyperplane bundle gives a relative evaluation map $\eta^*\Fc^\vee \arw \Oc_\eta(1)$, and the identity $\eta_*\Oc_\eta(1)=\Fc^\vee$.

\section{Birational map to a quadric fibration}\label{sec:birational_quadric_fibration}
Let us choose a five-dimensional vector space $V$, and fix a decomposition $V= P\oplus W$, where $\dim P=2$. This implies:
\begin{equation}\label{eq:plucker_space_decomposition}
\wedge^2 V=\wedge^2 P\oplus (P\otimes W)\oplus\wedge^2 W.
\end{equation}
We denote by $\sfG \coloneqq\Gr(2, V)$ the Grassmannian of two-dimensional subspaces, by $\sfH$ its Pl\"ucker class, and by $x=[\wedge^2 P]$ the distinguished point given by the direct sum decomposition of Equation~\eqref{eq:plucker_space_decomposition}. Put $A=(P\otimes W)\oplus\wedge^2 W$.

\begin{notation}\label{not:nodal_gm_presentation}
For $3\leq n\leq5$, choose $5-n$ general hyperplanes in $\PP\bigl(\wedge^2 V\bigr)$ passing through $x$, and a general degree-two homogeneous polynomial $q$ on $A$. Let $\sfL\subset \PP\bigl(\wedge^2 V\bigr)$ be the intersection of the $5-n$ hyperplane sections and let $\sfQ_x \subset \PP\bigl(\wedge^2 V\bigr)$ be the cone with vertex $x$ over $\VV(q)\subset\PP(A)$. We set
\begin{equation*}
\sfX=\sfG\cap\sfL\cap \sfQ_x.
\end{equation*}
\end{notation}
Then, $\sfX$ as in Notation~\ref{not:nodal_gm_presentation} is an ordinary GM variety of dimension $n$ with exactly one singular point, which is the ordinary double point $x$.

\subsection{The geometric resolution of $\sfX$}
Projection from $x$ gives a rational map $\pi_x:\sfG\dashrightarrow\PP(A)$. Choose coordinates $x_{ij}$ for $\wedge^2V$, with $x=[1:0:\cdots:0]$, the first coordinate being $x_{12}$, and put
\begin{equation*}
a=(x_{13},x_{14},x_{15}),\qquad b=(x_{23},x_{24},x_{25}),\qquad z=(x_{45},-x_{35},x_{34}).
\end{equation*}
The two Pl\"ucker equations not involving $x_{12}$ are
\begin{equation}\label{eq:projected_grassmannian}
Q_1=a\cdot z=0,\qquad Q_2=b\cdot z=0.
\end{equation}

\begin{lemma}\label{lem:grassmannian_projection_image}
The closure of the image of $\pi_x$ is $\overline\sfG=\VV(Q_1,Q_2)\subset\PP(A)$, which contains the linear space $\sfD=\PP(P\otimes W)$. Its singular locus is the Segre variety
\begin{equation*}
\sfM=\{z=0,\ \operatorname{rank}(a,b)=1\}\simeq\PP(P)\times\PP(W).
\end{equation*}
Along $\sfM$ it has a transverse three-dimensional nodal singularity.
\end{lemma}
\begin{proof}
In these coordinates, the five Pl\"ucker equations are
\begin{equation*}
\begin{aligned}
x_{12}z_1&=a_2b_3-a_3b_2,\\
x_{12}z_2&=a_3b_1-a_1b_3,\\
x_{12}z_3&=a_1b_2-a_2b_1,\\
a_1z_1+a_2z_2+a_3z_3&=0,\\
b_1z_1+b_2z_2+b_3z_3&=0.
\end{aligned}
\end{equation*}
Thus the image of $\pi_x$ is contained in $\VV(Q_1,Q_2)$. To prove the converse after taking the closure, we first consider the open set $z\ne0$, and then its complement $z=0$.

On the chart $z_3\ne0$, the last two equations give $a_3=-(a_1z_1+a_2z_2)/z_3$ and $b_3=-(b_1z_1+b_2z_2)/z_3$. Now take these expressions and plug them into the first two Pl\"ucker equations. Comparing the result with the third one we obtain:
\begin{equation*}
\begin{aligned}
a_2b_3-a_3b_2&=\frac{z_1}{z_3}\bigl(a_1b_2-a_2b_1\bigr),\\
a_3b_1-a_1b_3&=\frac{z_2}{z_3}\bigl(a_1b_2-a_2b_1\bigr).
\end{aligned}
\end{equation*}
Consequently, all three equations involving $x_{12}$ hold for the unique choice $x_{12}=(a_1b_2-a_2b_1)/z_3$. The same argument applies on the charts $z_1\ne0$ and $z_2\ne0$. Thus the two equations in Equation~\eqref{eq:projected_grassmannian} ensure that the three remaining Pl\"ucker equations admit a common, unique value of $x_{12}$. This reconstructs a point of $\sfG$ and gives an inverse to projection over $z\ne0$; in particular, projection is an isomorphism over this open set.

The complement $z=0$ is $\sfD$. Given a point $[a:b:0]$ of this space, choose a nonzero covector $z'$ annihilating both $a$ and $b$; such a covector exists since $\dim W=3$. The points $[a:b:\varepsilon z']$ belong to $\VV(Q_1,Q_2)$ for every $\varepsilon\in\CC$, and to the image of $\pi_x$ for $\varepsilon\ne0$. Their limit at $\varepsilon=0$ is $[a:b:0]$, proving the first claim. Notice that, when $z=0$, the first three Pl\"ucker equations require $a$ and $b$ to be proportional: points of $\sfD$ with $\operatorname{rank}(a,b)=2$ belong to the closure of the image, but not to the image of the rational map itself.

On each chart $z_i\ne0$, the two equations solve for $a_i$ and $b_i$, so $\overline\sfG$ has dimension six and is a complete intersection of two quadrics. With the coordinates ordered as $(a,b,z)$, its Jacobian matrix is
\begin{equation*}
\begin{pmatrix}z&0&a\\0&z&b\end{pmatrix}.
\end{equation*}
Its rows are independent whenever $z\ne0$. When $z=0$, its rank is less than two exactly when $a,b$ are proportional; the case $a=b=0$ is excluded in projective space. Writing $a=\lambda w$ and $b=\mu w$, we identify this locus with the image of the Segre embedding
\begin{equation*}
\PP(P)\times\PP(W)\arw\PP(P\otimes W),\qquad \bigl([\lambda:\mu],[w]\bigr)\longmapsto[\lambda w:\mu w].
\end{equation*}

For the local assertion, after interchanging the two rows and reordering the coordinates of $W$ if necessary, take the chart $a_1=1$ and write $t=b_1$. The first equation in~Equation~\eqref{eq:projected_grassmannian} eliminates $z_1=-a_2z_2-a_3z_3$, and the second becomes
\begin{equation}\label{eq:transverse_node}
u z_2+v z_3=0,\qquad u=b_2-ta_2,\quad v=b_3-ta_3.
\end{equation}
Thus this chart is the product of $\AA^3$, with coordinates $a_2,a_3,t$, and the hypersurface $\{uz_2+vz_3=0\}\subset\AA^4$. The locus $\sfM$ is given by $u=v=z_2=z_3=0$. The quadratic form $uz_2+vz_3$ is nondegenerate in the four remaining coordinates. This product description proves the asserted transverse nodal singularity.
\end{proof}

Let $\overline\sfL\subset\PP(A)$ be the image of $\sfL$ under the projection $\pi_x$. Then we introduce:
\begin{equation*}
\sfT=\overline \sfG\cap\overline\sfL\cap\VV(q),\qquad \sfZ=\sfM\cap\overline\sfL\cap\VV(q).
\end{equation*}
By definition, $\sfT$ is the closure of the image of $\sfX$, and $\sfZ$ is the image of the contracted locus, whereas its cone with vertex $x$ is the union of the lines through $x$ in $\sfX$. The hyperplanes restrict to $\Oc_{\sfM}(1,1)$ and the quadric to $\Oc_{\sfM}(2,2)$. Bertini's theorem gives a smooth intersection of dimension $n-3$, and
\begin{equation*}
\deg\sfZ=2\deg\sfM=6.
\end{equation*}
In particular, for $n=3$ it consists of six distinct points. Transversality in Lemma~\ref{lem:grassmannian_projection_image} and Bertini away from $\sfM$ show that $\sfT$ is smooth away from $\sfZ$, and has the local model in Equation~\eqref{eq:transverse_node} transverse to $\sfZ$.

Let us now consider the blowup $\varphi:\wt \sfX=\Bl_x\sfX\arw \sfX$, which resolves $\pi_x$. Write $\pi:\wt \sfX\arw \sfT$ for the resulting morphism. The preimage of the contracted locus of $\pi_x$ in $\wt \sfX$ is a $\PP^1$-bundle, whose bundle structure is contracted by $\pi$. Let us call $\alpha:\sfS_{\sfX}\arw\sfZ$ the projective bundle structure. In other words, the fibers of $\alpha$ are the strict transforms of the lines in $\sfX$ that pass through $x$. The exceptional divisor $\sfE$ of $\varphi$ is a smooth quadric of dimension $n-1$.

Let us summarize the projection, the resolution and its contracted locus with the following diagram:
\begin{equation*}
\begin{tikzcd}[column sep=large,row sep=large]
\sfE\ar[r,hook,"j"]\ar[d]&\wt\sfX\ar[swap, d,"\varphi"]\ar[dr,"\pi"]\ar[r, hookleftarrow] & \sfS_\sfX \ar[dr, "\alpha"] & \\
\{x\}\ar[r,hook] & \sfX\ar[r,dashed,"\pi_x"'] & \sfT\ar[r, hookleftarrow] & \sfZ.
\end{tikzcd}
\end{equation*}

\begin{lemma}\label{lem:exceptional_divisor_and_canonical_class}
If $h=\pi^*c_1\bigl(\Oc_{\sfT}(1)\bigr)$, then
\begin{equation*}
\sfE=\sfH-h,\qquad K_{\wt \sfX}=-(n-2)h.
\end{equation*}
Moreover, for a fiber $\sfC$ of $\alpha\colon \sfS_{\sfX}\arw\sfZ$ one has $\sfH\cdot \sfC=\sfE\cdot \sfC=1$ and $h\cdot \sfC=0$.
\end{lemma}
\begin{proof}
The projection is defined by hyperplanes through $x$, hence $h=\sfH-\sfE$. In $\Bl_x\sfG$ the proper transforms of the $5-n$ hyperplanes and of $\sfQ_x$ have classes $h$ and $2h$, respectively. Since $K_{\Bl_x\sfG}=-5\sfH+5\sfE=-5h$, adjunction gives the formula for the canonical class. Finally, $\sfC$ maps isomorphically to a Pl\"ucker line and is contracted by $\pi$, which gives the intersection numbers.
\end{proof}

\subsection{The quadric fibration}
Put $\sfB\coloneqq \Gr(2, W)\cong \PP\bigl(\wedge^2 W\bigr)$ and let $\Vc\subset W\otimes\Oc_{\sfB}$ be its tautological rank-two bundle. We denote the hyperplane class of this plane by $\zeta$ and its tangent bundle by $\Tc_{\sfB}$. One has $\Vc\simeq\Tc_{\sfB}^\vee(\zeta)$ and $\det\Vc\simeq\Oc_{\sfB}(-\zeta)$.

\begin{proposition}\label{prop:projected_grassmannian_bundle}
The proper transform of $\overline \sfG$ in $\Bl_{\sfD}\PP(A)$ is the projective bundle $\PP_{\sfB}(\Gc)$, where
\begin{equation*}
\Gc=(P\otimes\Vc)\oplus\wedge^2\Vc \simeq P\otimes\Tc_{\sfB}^\vee(\zeta)\oplus\Oc_{\sfB}(-\zeta).
\end{equation*}
Its relative hyperplane class is $h$. Under the induced small birational map from $\Bl_x\sfG$, the strict transform of $\zeta$ is $\sfH-2\sfE$.
\end{proposition}
\begin{proof}
We recall the construction of \cite[Proposition 2.10]{kuznetsov_prokhorov} in our notation. The blowup of $\PP(A)$ with center $\sfD$ resolves the projection $\PP(A)\dashrightarrow \PP(\wedge^2 W)$. This blowup is isomorphic to the projectivization of $P\otimes W\otimes \Oc_\sfB \oplus \Oc_\sfB(-\zeta)$. Now, the proper transforms of the equations $Q_1=Q_2=0$ of Equation~\eqref{eq:projected_grassmannian} are linear along the $\PP^6$-fibers of the projective bundle map: over a two-plane $U\subset W$, they require the two rows $a,b$ to lie in $U$. Equivalently, they are induced by the surjection
\begin{equation*}
(P\otimes W\otimes\Oc_{\sfB})\oplus\wedge^2\Vc\arw P\otimes\bigl((W\otimes\Oc_{\sfB})/\Vc\bigr),
\end{equation*}
which is zero on $\wedge^2\Vc$. Its kernel is $(P\otimes\Vc)\oplus\wedge^2\Vc=\Gc$. With our lines convention, the common zero locus is therefore $\PP_{\sfB}(\Gc)$, a $\PP^4$-bundle. The inclusion $\Gc\subset A\otimes\Oc_{\sfB}$ identifies its relative hyperplane class with $h$. On $\sfG$, the map to $\sfB$ is induced by the $2\times2$ minors of the projection $\Uc\arw W\otimes\Oc_{\sfG}$. These minors vanish to order two at $x$, so their strict transforms have class $\sfH-2\sfE$.

\end{proof}

The ambient blowup resolves the projection with center $\sfD$ as follows:
\begin{equation}\label{eq:ambient_projection_diagram}
\begin{tikzcd}[column sep=large,row sep=large]
\PP(\Nc_{\sfM|\sfD\cap\overline\sfG}^\vee)\ar[hook, r]\ar[d, "\PP^1"] & \PP_{\sfB}(\Gc)\ar[rr, bend left = 30pt, "\PP^4"]\ar[r,hook]\ar[d]&\Bl_{\sfD}\PP(A)\ar[swap, r, "\PP^6"]\ar[d]&\sfB\ar[d,equal]\\
\sfM \ar[hook, r] & \overline\sfG\ar[r,hook]&\PP(A)\ar[r,dashed]&\sfB.
\end{tikzcd}
\end{equation}

\begin{lemma}\label{lem:linear_section_bundle}
The projected hyperplanes induce a surjective map
\begin{equation}\label{eq:linear_section_bundle_kernel}
0\arw\Ec^\vee\arw\Gc \xrightarrow{\nu} K^\vee\otimes\Oc_{\sfB}\arw0, \qquad \dim K=5-n.
\end{equation}
Here $K$ is the space of the chosen hyperplane equations. The kernel of $\nu$ has rank $n$ and determinant $\Oc_{\sfB}(-3\zeta)$. The proper transform of $\overline \sfG\cap\overline\sfL$ is $\PP_{\sfB}\bigl(\Ec^\vee\bigr)$.
\end{lemma}
\begin{proof}
The dual of $\Gc$ is generated by the restrictions of linear forms on $A$. Thus at each point the parameters for $\nu$ map surjectively to all matrices of size $(5-n)\times5$. Rank drop has codimension $n+1>2$, so a general choice has empty degeneracy locus on $\sfB$. The fiber description in the preceding proposition identifies the proper transform with the projectivized kernel. Taking determinants gives $\det\Ec^\vee=\det\Gc=\Oc_{\sfB}(-3\zeta)$.
\end{proof}

\begin{lemma}\label{lem:resolved_quadric_fibration}
The proper transform $\sfR$ of $\sfT$ is the zero locus of a section of $\Oc_{\PP_{\sfB}\bigl(\Ec^\vee\bigr)}(2h)$. The resulting map $\psi:\sfR\arw \sfB$ is a flat quadric fibration of relative dimension $n-2$. Its discriminant $\Delta\subset \sfB$ is a smooth sextic, all singular fibers have corank one, and
\begin{equation*}
K_{\sfR}=-(n-2)h.
\end{equation*}
The map $\rho:\sfR\arw \sfT$ is a small resolution with exceptional locus a $\PP^1$-bundle $\beta:\sfS_{\sfR}\arw\sfZ$.
\end{lemma}
\begin{proof}
The equation is the restriction of $q$ to $\Ec^\vee$, that is, a symmetric map $\Ec^\vee\arw\Ec$. In the space of symmetric matrices, corank at least two has codimension three, and the zero matrix has codimension $n(n+1)/2$. Let $S=\Sym^2 A^\vee$ be the space of ambient quadratic equations. The evaluation morphism is the vector-bundle map $\operatorname{ev}:S\otimes\Oc_{\sfB}\arw\Sym^2\Ec$, whose fiber at $b\in\sfB$ sends $q$ to $q|_{\Ec_b^\vee}$. Since $\Ec_b^\vee\subset A$, every quadratic form on this subspace extends to $A$, so $\operatorname{ev}$ is surjective. Thus the induced morphism of total spaces $S\times\sfB\arw\operatorname{Tot}(\Sym^2\Ec)$ is smooth and surjective. Generic smoothness and these codimension bounds show that a general $q$ avoids both loci on the surface $\sfB$, and meets the corank-one locus transversely. This proves flatness and the smoothness assertion. Bertini on the smooth projective bundle proves that $\sfR$ itself is smooth. The determinant of the symmetric map is a section of $(\det\Ec)^{\otimes2}\simeq\Oc_{\sfB}(6\zeta)$.

Over a point of $\sfT\setminus\sfZ$, the two-plane in $W$ is uniquely determined. Over a point $(\ell,w)\in\sfZ\subset\PP(P)\times\PP(W)$, it is any two-plane containing $w$, parametrized by $\PP(W/w)\simeq\PP^1$. In the local Equation~\eqref{eq:transverse_node}, the morphism is the blowup of the divisorial ideal $(z_2,z_3)$, one of the two small resolutions of the node. Finally, the canonical class formula for a projective bundle gives
\begin{equation*}
K_{\PP_{\sfB}\bigl(\Ec^\vee\bigr)}=-nh+ \eta^*(K_{\sfB}-c_1(\Ec^\vee))=-nh,
\end{equation*}
where $\eta:\PP_{\sfB}\bigl(\Ec^\vee\bigr)\arw \sfB$ is the projection. Adjunction proves the stated formula.
\end{proof}

Summing all up, we have the following diagram:
\begin{equation*}
\begin{tikzcd}[column sep=huge,row sep=large]
\sfS_\sfR \ar[r, hook]\ar[d, "\beta"] &  \sfR\ar[r,hook]\ar[d,"\rho"']\ar[dr,"\psi" description]&\PP_{\sfB}\bigl(\Ec^\vee\bigr)\ar[d,"\eta"]\\
\sfZ \ar[r, hook] & \sfT\ar[r,dashed]&\sfB.
\end{tikzcd}
\end{equation*}
The bottom dashed arrow is the restriction of the rational projection of Diagram~\eqref{eq:ambient_projection_diagram}.

\begin{assumption}\label{ass:generality}
We assume that the data of Notation~\ref{not:nodal_gm_presentation} satisfy the following conditions:
\begin{enumerate}[label=\textup{(\roman*)}]
\item The linear section $\sfG\cap\sfL$ is smooth.
\item The variety $\sfX$ has precisely one singular point, namely the ordinary double point $x$.
\item The intersections defining $\sfZ=\sfM\cap\overline\sfL\cap\VV(q)$ are transverse.
\item The bundle map $\nu$ in~Equation~\eqref{eq:linear_section_bundle_kernel} is surjective.
\item The quadratic form defining $\psi:\sfR\arw \sfB$ has smooth discriminant $\Delta$, and all its singular fibers have corank one.
\end{enumerate}
\end{assumption}

These conditions hold on a nonempty open subset of the data, as verified above. The constructions and geometric assertions above also apply to any individual presentation satisfying these conditions.

The qualifiers in Assumption~\ref{ass:generality} concern the geometry of the projections, and not just the local analytic type of $x$. We do not identify this open locus with the entire locus of one-nodal GM varieties. It does contain general nodal ordinary GM varieties. To see that the cone presentation loses no generality, write an ambient quadratic equation through $x$ as $x_{12}L+q_0$, where $L$ and $q_0$ involve the other nine coordinates. The condition that its restriction to the smooth linear section of $\sfG$ is singular at $x$ says that $L$ belongs to the span of the linear-section equations and $x_{34},x_{35},x_{45}$. Subtracting multiples of $x_{12}$ times the linear equations, followed by multiples of the three Pl\"ucker equations involving $x_{12}$, removes $x_{12}L$ without changing the intersection. The resulting equation depends only on $A$. Since this operation is a surjective linear parametrization of the restricted equations with a double point at $x$, general such equations are obtained from general $q$ on $A$.

\subsection{A relative Atiyah flop}
The two small resolutions above resolve $\sfT$ along its singular locus $\sfZ$. One records the line through $x$, and the other records the two-plane containing its $W$-direction. The incidence description in \cite[Theorem 2.2 and Section 2.3]{kuznetsov_prokhorov} identifies their graph with the blowup of either exceptional locus. Restriction to our transverse linear and quadratic sections commutes with these blowups: locally these sections eliminate parameters along $\sfM$ and leave the same node Equation~\eqref{eq:transverse_node}. Consequently, the common resolution is
\begin{equation*}
\sfY=\Bl_{\sfS_{\sfX}}\wt\sfX=\Bl_{\sfS_{\sfR}}\sfR,
\end{equation*}
with blowup maps $p:\sfY\arw\wt\sfX$ and $q:\sfY\arw\sfR$. Its exceptional divisor is $\sfF=\sfS_{\sfX}\times_{\sfZ}\sfS_{\sfR}$; we denote its inclusion by $\iota:\sfF\hookrightarrow\sfY$. For each $\gamma\in\sfZ$, the fiber $\sfF_\gamma$ is isomorphic to $\PP^1\times\PP^1$, and the normal line bundle of $\sfF$ restricts to $\Oc(-1,-1)$ on $\sfF_\gamma$. Thus the map $\mu:\wt\sfX\dashrightarrow\sfR$ is a relative Atiyah flop; this description, in addition to equality of canonical classes, specifies its local model. For $n=3$ the normal bundle of each contracted curve is $\Oc_{\PP^1}(-1)^{\oplus2}$.

The flop, its two blowup centers, and the original birational construction are recorded in the following diagram:
\begin{equation}\label{eq:flop_diagram}
\begin{gathered}
\begin{tikzcd}[column sep={1.25cm,between origins},row sep={0.85cm,between origins}]
&&&&\sfF\ar[ddddllll,"{p|_{\sfF}}"']\ar[ddddrrrr,"{q|_{\sfF}}"]\ar[dd,hook,"\iota"]&&&&\\
&&&&&&&&\\
&&&&\sfY\ar[ddll,"p"']\ar[ddrr,"q"]&&&&\\
&&&&&&&&\\
\sfS_{\sfX}\ar[rr,hook]\ar[ddddrrrr,"\alpha"']&&\wt\sfX\ar[rrrr,dashed,"\mu"]\ar[ddrr,"\pi"']\ar[d,"\varphi"']&&&&\sfR\ar[ddll,"\rho"]\ar[d,"\psi"]&&\sfS_{\sfR}\ar[ll,hook']\ar[ddddllll,"\beta"]\\[4mm]
&&\sfX&&&&\sfB&&\\
&&&&\sfT&&&&\\
&&&&&&&&\\
&&&&\sfZ\ar[uu,hook]&&&&
\end{tikzcd}
\\[1em]
\sfF=\sfS_{\sfX}\times_{\sfZ}\sfS_{\sfR},\qquad \sfF_\gamma\simeq\PP^1\times\PP^1\quad\bigl(\gamma\in\sfZ\bigr).
\end{gathered}
\end{equation}
The outer diamond is cartesian, and each of its four arrows is a $\PP^1$-bundle. Its two composites give the $\PP^1\times\PP^1$-fibration $\sfF\arw\sfZ$. The arrows from the outer diamond to the inner one are the natural inclusions; in particular, $p^{-1}(\sfS_{\sfX})=\sfF=q^{-1}(\sfS_{\sfR})$.

\begin{lemma}\label{lem:common_blowup_divisor_classes}
Consider the divisors $\zeta_{\sfX}=\sfH-2\sfE$ on $\wt \sfX$ and $\zeta_{\sfR}=\psi^*\zeta$ on $\sfR$. On $\sfY$ one has
\begin{equation*}
p^*h=q^*h,\qquad p^*\zeta_{\sfX}=q^*\zeta_{\sfR}+\sfF.
\end{equation*}
On $\sfF_\gamma$, with the first factor mapping to $\sfS_{\sfX}$, the restrictions of $p^*\zeta_{\sfX}$ and $q^*\zeta_{\sfR}$ are the classes of $\Oc(-1,0)$ and $\Oc(0,1)$.
\end{lemma}
\begin{proof}
Both $p^*h$ and $q^*h$ are pullbacks from $\sfT$, hence they are linearly equivalent on $\sfY$. The classes $\zeta_{\sfX}$ and $\zeta_{\sfR}$ agree off the exceptional locus, so their pullbacks differ by a divisor class supported on $\sfF$. On a fiber of $q|_{\sfF}$, by Lemma~\ref{lem:exceptional_divisor_and_canonical_class}, the degree of $p^*\zeta_{\sfX}$ is $\sfH\cdot \sfC-2\sfE\cdot \sfC=-1$, whereas that of $\sfF$ is $-1$. This determines the coefficient to be one on every component. A fiber of $\sfS_{\sfR}\arw\sfZ$ maps to the line of two-planes in $W$ containing a fixed vector, hence has $\zeta_{\sfR}$-degree one, giving the stated restrictions.
\end{proof}

\begin{remark}\label{rem:nodal_gm_surfaces}
The same construction for $n=2$ uses three general hyperplanes. The map $\nu$ is still surjective, since its expected rank-drop codimension is three. Now $\sfZ_2$ is empty and both small modifications are isomorphisms. The rank-two quadratic form has no zero fiber, again by codimension three; hence $\sfR_2\arw\PP^2$ is a finite flat double cover with smooth sextic branch curve. Its canonical bundle is trivial and $H^1(\Oc_{\sfR_2})=0$, so $\sfR_2\simeq\wt \sfX_2$ is a $K3$ surface. The new ample polarization is $\sfH-2\sfE$ and we obtain the following identities:
\begin{equation*}
\sfH^2=10,\quad \sfH\cdot \sfE=0,\quad \sfE^2=-2,\qquad (\sfH-2\sfE)^2=2.
\end{equation*}
\end{remark}

\begin{remark}
Special GM varieties are double covers of linear sections of $\sfG$. Although the projections above suggest related constructions for their covers, the picture becomes considerably more complex, and diverges substantially at the level of derived categories. In the remainder of this paper, all statements about GM varieties will concern ordinary varieties, and we will leave special GM varieties to future work.
\end{remark}

\section{The categorical resolution of the Kuznetsov component}\label{sec:categorical_resolution}
\subsection{A reminder on semiorthogonal decompositions}
Let us briefly recall some notation about mutations of exceptional objects. We refer to \cite{huyb-book-FM} for the general formalism and for a thorough introduction to the subject. If $\Tc=\langle\Ac,\Bc\rangle$ is a semiorthogonal decomposition into admissible subcategories, then
\begin{equation*}
\langle\Ac,\Bc\rangle=\langle\LL_\Ac\Bc,\Ac\rangle =\langle\Bc,\RR_\Bc\Ac\rangle.
\end{equation*}
For a pair of exceptional objects $\langle\Ac^\bullet,\Bc^\bullet\rangle$ there is the following explicit description of the left and right mutations:
\begin{equation*}
\begin{split}
\LL_{\Ac^\bullet} \Bc^\bullet&=\operatorname{Cone}\bigl(\mathbf{R}\operatorname{Hom}\bigl(\Ac^\bullet,\Bc^\bullet\bigr)\otimes \Ac^\bullet\arw \Bc^\bullet\bigr),\\
\RR_{\Bc^\bullet} \Ac^\bullet&=\operatorname{Cone}\bigl(\Ac^\bullet\arw \mathbf{R}\operatorname{Hom}\bigl(\Ac^\bullet,\Bc^\bullet\bigr)^\vee\otimes \Bc^\bullet\bigr)[-1].
\end{split}
\end{equation*}
In particular, the cone of a surjective evaluation map is its kernel shifted by $[1]$. Shifts do not change the subcategory generated by an exceptional object.

For ordered blocks of objects, we use the abbreviation
\begin{equation*}
\bigl\langle\Ac_{j,1}^\bullet,\ldots,\Ac_{j,s}^\bullet\bigr\rangle_{j=a}^{b} \coloneqq\bigl\langle\Ac_{a,1}^\bullet,\ldots,\Ac_{a,s}^\bullet,\ldots,\Ac_{b,1}^\bullet,\ldots,\Ac_{b,s}^\bullet\bigr\rangle.
\end{equation*}
Thus the blocks are concatenated in increasing order of $j$, with the objects in each block kept in the displayed order. An empty index range contributes no objects.

\subsection{The Bondal--Orlov flop functor}
Throughout this section we work under the geometric hypotheses of Assumption~\ref{ass:generality}. We keep the notation of Diagram~\eqref{eq:flop_diagram}, and write
\begin{equation*}
\Phi=\mathbf{R}q_*\mathbf{L}p^*:\dbcoh\bigl(\wt \sfX\bigr)\arw\dbcoh(\sfR).
\end{equation*}
We recall the classes
\begin{equation*}
\zeta_{\sfX}=\sfH-2\sfE=2h-\sfH\quad\hbox{on }\wt \sfX, \qquad \zeta_{\sfR}=\psi^*c_1(\Oc_{\PP^2}(1))\quad\hbox{on }\sfR.
\end{equation*}
They correspond by strict transform across the flop. The class $h$, on the other hand, is pulled back from the common contraction on both sides.

The following is a simple application of the Bondal--Orlov theorem on flops \cite{bondal_orlov_flop}, in the relative form of \cite[Theorem A$'$ and Section 2]{addington_donovan_meachan_flops}.

\begin{proposition}\label{prop:relative_flop_equivalence}
The functor $\Phi$ is an equivalence. Moreover,
\begin{equation*}
\Phi\bigl(\Gc^\bullet\otimes\Oc(kh)\bigr)\simeq\Phi\bigl(\Gc^\bullet\bigr)\otimes\Oc(kh) \qquad(k\in\ZZ).
\end{equation*}
\end{proposition}
\begin{proof}
The centers $\sfS_{\sfX}$ and $\sfS_{\sfR}$ are projective bundles over the smooth variety $\sfZ$, and their common exceptional divisor is $\sfS_{\sfX}\times_{\sfZ} \sfS_{\sfR}$. The normal bundle of $\sfS_{\sfX}$ has relative degree $-1$ on its projective line fibers and rank two; equivalently it has the form $\Oc_{\sfS_{\sfX}}(-1)\otimes\alpha^*\Wc$ for a rank-two bundle $\Wc$ on $\sfZ$. Indeed, after tensoring by a relative $\Oc(1)$, its restriction to every fiber is trivial, so cohomology and base change gives this description. Thus this is a relative standard flop. The relative Bondal--Orlov equivalence of \cite[Theorem A$'$ and Section 2]{addington_donovan_meachan_flops} applies to $\mathbf{R}q_*\mathbf{L}p^*$. The last assertion is the projection formula, since $p^*h=q^*h$ on $\sfY$.
\end{proof}

\subsection{Mutations inside $\dbcoh\bigl(\wt \sfX\bigr)$}
Set $r=n-2$. We first specify the categorical resolution used in the statement. The notation $\Uc$ denotes the tautological bundle restricted from $\Gr(2,5)$, as well as its pullback to $\wt \sfX$.

\begin{proposition}\label{prop:resolved_gm_semiorthogonal_decomposition}
The sequence
\begin{equation*}
\Pc=\langle\Oc,\Uc^\vee,\ldots,\Oc\bigl((r-1)\sfH\bigr),\Uc^\vee\bigl((r-1)\sfH\bigr)\rangle
\end{equation*}
is an exceptional collection on $\sfX$. Define
\begin{equation*}
\Kc u(\sfX)\coloneqq \Pc^\perp\subset\dbcoh(\sfX), \qquad \Kc u(\sfX)^{\mathrm{perf}}\coloneqq \Kc u(\sfX)\cap\operatorname{Perf}(\sfX).
\end{equation*}
There is a semiorthogonal decomposition
\begin{equation}\label{eq:resolved_gm_semiorthogonal_decomposition}
\begin{split}
\dbcoh\bigl(\wt \sfX\bigr)=\langle&j_*\Oc_{\sfE}(-rh),\ldots,j_*\Oc_{\sfE}(-h), \wt{\Kc u}(\sfX),\\
&\Oc,\Uc^\vee,\ldots, \Oc\bigl((r-1)\sfH\bigr),\Uc^\vee\bigl((r-1)\sfH\bigr)\rangle.
\end{split}
\end{equation}
The residual category $\wt{\Kc u}(\sfX)$ is a categorical resolution of $\Kc u(\sfX)$: the restrictions
\begin{equation*}
\mathbf{L}\varphi^*:\Kc u(\sfX)^{\mathrm{perf}}\arw\wt{\Kc u}(\sfX), \qquad \mathbf{R}\varphi_*:\wt{\Kc u}(\sfX)\arw\Kc u(\sfX)
\end{equation*}
are adjoint, with $\mathbf{L}\varphi^*$ left adjoint, and $\mathbf{R}\varphi_*\mathbf{L}\varphi^*\simeq\mathrm{id}$.
\end{proposition}
\begin{proof}
We give a uniform argument for the three dimensions. We first explain why the node does not affect the cohomology computation. Let $i:\sfX\hookrightarrow\sfG$ be the inclusion. The defining hyperplanes and quadric give a section of
\begin{equation*}
\Fc=\Oc_{\sfG}(\sfH)^{\oplus(5-n)}\oplus\Oc_{\sfG}(2\sfH),
\end{equation*}
whose zero locus has codimension $6-n=\operatorname{rk}(\Fc)$. The local rings of the smooth Grassmannian are Cohen--Macaulay, so these equations form a regular sequence at every point of $\sfX$, including $x$; see \cite[Tag 02JN]{stacks_project}. Consequently, their Koszul complex $\Kc^\bullet$ is a locally free resolution of $i_*\Oc_{\sfX}$, even though $\sfX$ is singular \cite[Tag 062F]{stacks_project}:
\begin{equation*}
0\arw\wedge^{6-n}\Fc^\vee\arw\cdots\arw\Fc^\vee\arw\Oc_{\sfG}\arw i_*\Oc_{\sfX}\arw0.
\end{equation*}
Moreover, all the objects in $\Pc$ are restrictions of vector bundles on $\sfG$. For any two such bundles $\Gc,\Hc$, local freeness and the projection formula give
\begin{equation*}
\mathbf{R}\operatorname{Hom}_{\sfX}\bigl(i^*\Gc,i^*\Hc\bigr) \simeq \mathbf{R}\Gamma\bigl(\sfG,\Gc^\vee\otimes\Hc\otimes\Kc^\bullet\bigr).
\end{equation*}
Thus the required Ext groups are computed by the cohomology of the terms of a complex on $\sfG$. The Borel--Weil--Bott theorem is applied on this smooth Grassmannian. In particular, it gives the following vanishings (see \cite[Lemma 3.2]{kapranov_homogeneous_spaces}, together with Serre duality):
\begin{equation}\label{eq:ambient_bwb_cohomology}
\begin{aligned}
H^\bullet\bigl(\sfG,\Oc_{\sfG}(-k\sfH)\bigr)&=0
\qquad 1\leq k\leq4,\\
H^\bullet\bigl(\sfG,\Uc^\vee(-k\sfH)\bigr)&=0
\qquad 1\leq k\leq5,\\
H^\bullet\bigl(\sfG,\operatorname{Sym}^2\Uc^\vee(-k\sfH)\bigr)&=0
\qquad 1\leq k\leq5.
\end{aligned}
\end{equation}
Also $\mathbf{R}\Gamma(\sfG,\Oc)=\CC$ and $\mathbf{R}\Gamma(\sfG,\Uc^\vee)= V^\vee$, concentrated in degree zero. Tensoring the Koszul resolution by the bundles in question proves all the required vanishings on $\sfX$. More explicitly, backward morphisms between the displayed objects are cohomology groups of
\begin{equation*}
\Oc(-m\sfH),\quad\Uc^\vee(-m\sfH),\quad \Uc^\vee\bigl(-(m+1)\sfH\bigr),\quad \operatorname{Sym}^2\Uc^\vee\bigl(-(m+1)\sfH\bigr),
\end{equation*}
with $1\leq m\leq r-1$, together with $\Uc^\vee(-\sfH)$ for the pair in a single block. Exceptionality of $\Uc^\vee$ uses
\begin{equation*}
\mathcal{E}nd(\Uc^\vee)=\Oc\oplus\operatorname{Sym}^2\Uc^\vee(-\sfH),
\end{equation*}
so the case $k=1$ in the last line of Equation~\eqref{eq:ambient_bwb_cohomology} is included. Every nonconstant term of the relevant Koszul complexes is acyclic by Equation~\eqref{eq:ambient_bwb_cohomology}. This proves that $\Pc$ is an exceptional collection.

The node is rational, hence $\mathbf{R}\varphi_*\Oc_{\wt \sfX}=\Oc_{\sfX}$. The projection formula therefore proves that pullback preserves these morphism spaces. The pulled-back collection is exceptional on $\wt \sfX$.

The exceptional divisor is the smooth quadric $\sfE\simeq \sfQ^{n-1}$, with normal bundle $\Oc_{\sfE}(-h)$. The divisor adjunction triangle shows that the self-Ext groups of $j_*\Oc_{\sfE}(-ah)$ are computed by the cohomology of $\Oc_{\sfE}$ and $\Oc_{\sfE}(-h)$, with the usual shift on the normal-bundle term. For backward morphisms between two of these objects, the two terms are $\Oc_{\sfE}(-kh)$ and $\Oc_{\sfE}(-(k+1)h)$ with $1\leq k\leq r-1$. Since
\begin{equation*}
H^\bullet\bigl(\sfE,\Oc_{\sfE}(-kh)\bigr)=0\qquad(1\leq k\leq r),
\end{equation*}
the torsion objects in Equation~\eqref{eq:resolved_gm_semiorthogonal_decomposition} form an exceptional collection. For every pulled-back bundle $\Gc$ in $\Pc$,
\begin{equation*}
\mathbf{R}\operatorname{Hom}_{\wt \sfX}(\varphi^*\Gc,j_*\Oc_{\sfE}(-kh)) =\Gc|_x^\vee\otimes \mathbf{R}\Gamma(\sfE,\Oc_{\sfE}(-kh))=0.
\end{equation*}
The collections consequently have an admissible residual category, which defines $\wt{\Kc u}(\sfX)$ in Equation~\eqref{eq:resolved_gm_semiorthogonal_decomposition}.

It remains to check the resolution property. If $\Ac^\bullet\in\Kc u(\sfX)^{\mathrm{perf}}$, then
\begin{equation*}
\mathbf{R}\operatorname{Hom}\bigl(\mathbf{L}\varphi^*\Ac^\bullet,j_*\Oc_{\sfE}(-kh)\bigr)=0, \qquad \mathbf{R}\operatorname{Hom}\bigl(\varphi^*\Gc,\mathbf{L}\varphi^*\Ac^\bullet\bigr)=0 \quad(\Gc\in\Pc).
\end{equation*}
The first vanishing follows by pushforward of $j_*\Oc_{\sfE}(-kh)$, and the second by the projection formula and the definition of $\Kc u(\sfX)$. Thus $\mathbf{L}\varphi^*\Ac^\bullet$ lies in the stated residual category. Conversely, for $\Bc^\bullet\in\wt{\Kc u}(\sfX)$, adjunction gives
\begin{equation*}
\mathbf{R}\operatorname{Hom}\bigl(\Gc,\mathbf{R}\varphi_*\Bc^\bullet\bigr)=\mathbf{R}\operatorname{Hom}\bigl(\varphi^*\Gc,\Bc^\bullet\bigr)=0,
\end{equation*}
so $\mathbf{R}\varphi_*\Bc^\bullet\in\Kc u(\sfX)$. The adjunction restricts, and its unit is an isomorphism by rationality of the node. Finally the residual category is smooth and proper, being admissible in the derived category of the smooth projective variety $\wt \sfX$. These are the categorical-resolution properties of \cite{kuznetsov_resolutions_of_singularities}.
\end{proof}

Let us now compute some mutations that will be used later in this section.
\begin{lemma}\label{lem:right_mutation_exceptional_divisor}
Let $\Gc$ be a vector bundle on $\sfX$. Then
\begin{equation*}
\RR_{j_*\Oc_{\sfE}}\varphi^*\Gc\simeq\varphi^*\Gc(-\sfE).
\end{equation*}
More generally, for integers $a,b$, the right mutation of $\varphi^*\Gc(a\sfH-b\sfE)$ through $j_*\Oc_{\sfE}(bh)$ is $\varphi^*\Gc(a\sfH-(b+1)\sfE)$.
\end{lemma}
\begin{proof}
Restriction to $\sfE$ identifies $\varphi^*\Gc(a\sfH-b\sfE)|_{\sfE}$ with $\Gc|_x\otimes\Oc_{\sfE}(bh)$. Consequently the morphism space to $j_*\Oc_{\sfE}(bh)$ is $\Gc|_x^\vee$ in degree zero, and its coevaluation is the restriction map. The assertion follows from
\begin{equation*}
0\arw\varphi^*\Gc(a\sfH-(b+1)\sfE) \arw\varphi^*\Gc(a\sfH-b\sfE) \arw j_*\bigl(\Gc|_x\otimes\Oc_{\sfE}(bh)\bigr) \arw0
\end{equation*}
and the convention for right mutation.
\end{proof}

\begin{lemma}\label{lem:left_mutation_exceptional_divisor}
One has $\LL_{\Oc}j_*\Oc_{\sfE}\simeq\Oc(-\sfE)[1]$.
\end{lemma}
\begin{proof}
The evaluation triangle is the triangle associated with the short exact sequence
\begin{equation*}
0\arw\Oc(-\sfE)\arw\Oc\arw j_*\Oc_{\sfE}\arw0.
\end{equation*}
\end{proof}

\begin{lemma}\label{lem:tautological_evaluation_mutation}
Put $\Ac_{\sfX}^\bullet\coloneqq\LL_{\Oc}\Uc^\vee(-\sfE)$. Then:
\begin{equation}\label{eq:tautological_evaluation_complex}
\Ac_{\sfX}^\bullet \cong \operatorname{Cone}\bigl(W^\vee\otimes\Oc\arw\Uc^\vee(-\sfE)\bigr).
\end{equation}
In particular, the evaluation map $W^\vee\otimes\Oc\arw\Uc^\vee(-\sfE)$ is not surjective along $\sfS_{\sfX}$.
\end{lemma}
\begin{proof}
The preceding Koszul computation gives $\mathbf{R}\Gamma(\wt \sfX,\Uc^\vee)= V^\vee$. In the restriction sequence
\begin{equation*}
0\arw\Uc^\vee(-\sfE)\arw\Uc^\vee \arw j_*\bigl(P^\vee\otimes\Oc_{\sfE}\bigr)\arw0,
\end{equation*}
the map on global sections is the surjection $V^\vee\arw  P^\vee$. Hence $\mathbf{R}\Gamma(\wt \sfX,\Uc^\vee(-\sfE))= W^\vee$ in degree zero, proving Equation~\eqref{eq:tautological_evaluation_complex}. Now, if $\sfC$ is a fiber of $\sfS_{\sfX}\arw\sfZ$, then
\begin{equation*}
\Uc^\vee(-\sfE)|_{\sfC}=\Oc_{\sfC}(-1)\oplus\Oc_{\sfC}.
\end{equation*}
This bundle cannot be a quotient of a trivial bundle, and therefore the map $W^\vee\otimes\Oc\arw\Uc^\vee(-\sfE)$ cannot be surjective.
\end{proof}

\begin{proposition}\label{prop:mutated_gm_decomposition}
Write $\Kc'=\RR_{\Oc(-\sfE)}\wt{\Kc u}(\sfX)$. Then
\begin{equation}\label{eq:mutated_gm_decomposition}
\dbcoh\bigl(\wt \sfX\bigr)= \left\langle\Kc', \bigl\langle\Ac_{\sfX}^\bullet(kh),\Oc(kh),\Oc(\zeta_{\sfX}+kh)\bigr\rangle_{k=0}^{r-1} \right\rangle,
\end{equation}
where the inner blocks are ordered by increasing $k$.
\end{proposition}
\begin{proof}
The inverse Serre functor sends the torsion objects $j_*\Oc_\sfE(kh)$ in Equation~\eqref{eq:resolved_gm_semiorthogonal_decomposition} to the end. Since $K_{\wt \sfX}=-rh$, we obtain, omitting shifts of exceptional objects,
\begin{equation*}
\left\langle\wt{\Kc u}(\sfX), \bigl\langle\Oc(j\sfH),\Uc^\vee(j\sfH)\bigr\rangle_{j=0}^{r-1}, j_*\Oc_{\sfE},j_*\Oc_{\sfE}(h),\ldots,j_*\Oc_{\sfE}((r-1)h)\right\rangle.
\end{equation*}
Move $j_*\Oc_{\sfE}$ left past all bundles except the first $\Oc$, using right mutations of those bundles. By Lemma~\ref{lem:right_mutation_exceptional_divisor}, this replaces $\Uc^\vee$ by $\Uc^\vee(-\sfE)$ and, for each $1\leq j\leq r-1$, replaces the pair $\Oc(j\sfH),\Uc^\vee(j\sfH)$ by $\Oc(j\sfH-\sfE),\Uc^\vee(j\sfH-\sfE)$. The object $j_*\Oc_{\sfE}$ now sits immediately after the first $\Oc$. Left-mutating it through that $\Oc$ replaces the ordered pair $\Oc,j_*\Oc_{\sfE}$ by $\Oc(-\sfE),\Oc$, by Lemma~\ref{lem:left_mutation_exceptional_divisor}, again omitting a shift. Thus the first two bundles $\Oc,\Uc^\vee$ and the object $j_*\Oc_{\sfE}$ supported on the exceptional divisor have been replaced by the ordered triple
\begin{equation*}
\Oc(-\sfE),\ \Oc,\ \Uc^\vee(-\sfE).
\end{equation*}
The objects $j_*\Oc_{\sfE}(kh)$ with $1\leq k\leq r-1$ are unchanged.

This operation iterates. At the $k$th step, the unchanged pairs are $\Oc(j\sfH-k\sfE),\Uc^\vee(j\sfH-k\sfE)$ for $j\geq k$. Their restrictions to $\sfE$ are trivial bundles tensored by $\Oc_{\sfE}(kh)$. Moving $j_*\Oc_{\sfE}(kh)$ left to the first such pair and then applying the two lemmas creates
\begin{equation*}
\Oc(kh-\sfE),\ \Oc(kh),\ \Uc^\vee(kh-\sfE),
\end{equation*}
and subtracts one further $\sfE$ from each later pair. After $r$ steps, the decomposition is
\begin{equation*}
\left\langle\wt{\Kc u}(\sfX), \bigl\langle\Oc(kh-\sfE),\Oc(kh),\Uc^\vee(kh-\sfE)\bigr\rangle_{k=0}^{r-1} \right\rangle.
\end{equation*}
Right-mutate the residual category through $\Oc(-\sfE)$ and rotate that line bundle to the end. Its new twist is $\Oc(rh-\sfE)$. As $h-\sfE=\zeta_{\sfX}$, the exceptional objects now group into the triples
\begin{equation*}
\Oc(kh),\ \Uc^\vee(kh-\sfE),\ \Oc(\zeta_{\sfX}+kh), \qquad 0\leq k\leq r-1.
\end{equation*}
Finally left-mutate the second object of each triple through the first. Equation~\eqref{eq:tautological_evaluation_complex}, tensored by $\Oc(kh)$, gives Equation~\eqref{eq:mutated_gm_decomposition}.
\end{proof}

\subsection{The evaluation complex across the flop}
To prove Theorem~\ref{thm:intro_kuznetsov_clifford_equivalence}, the remaining step is to compute the image of the objects from Proposition~\ref{prop:mutated_gm_decomposition} in $\sfR$, i.e. through the Bondal--Orlov flop functor. Let $\Vc\subset  W\otimes\Oc_{\Gr(2, W)}$ be the rank-two tautological bundle. We identify $\Gr(2, W)$ with the base $\PP^2$ of $\psi$.

\begin{lemma}\label{lem:tautological_flop_transform}
There are natural isomorphisms
\begin{equation*}
\Phi\bigl(\Uc^\vee(-\sfE)\bigr)\simeq\psi^*\Vc^\vee, \qquad \Phi\bigl(\Ac_{\sfX}^\bullet\bigr)\simeq\Oc_{\sfR}(-\zeta_{\sfR})[1].
\end{equation*}
Moreover, for every $m\in\ZZ$
\begin{equation*}
\Phi(\Oc(mh))=\Oc(mh), \quad \Phi\bigl(\Oc(mh+\zeta_{\sfX})\bigr)=\Oc(mh+\zeta_{\sfR}).
\end{equation*}
\end{lemma}

\begin{proof}
Since $q$ is the blowup of a smooth center, $\mathbf{R}q_*\Oc_{\sfY}\simeq\Oc_{\sfR}$. The equality $p^*h=q^*h$ and the projection formula give $\Phi(\Oc(mh))\simeq\Oc(mh)$ for every $m\in\ZZ$.

By Lemma~\ref{lem:common_blowup_divisor_classes},
\begin{equation}\label{eq:tautological_transform_divisor_relation}
p^*\zeta_{\sfX}=q^*\zeta_{\sfR}+\sfF.
\end{equation}
The divisor sequence
\begin{equation*}
0\arw\Oc_{\sfY} \arw\Oc_{\sfY}(\sfF) \arw\iota_*\Oc_{\sfF}(\sfF) \arw0
\end{equation*}
gives $\mathbf{R}q_*\Oc_{\sfY}(\sfF)\simeq\Oc_{\sfR}$: the restriction of $\Oc_{\sfF}(\sfF)$ to every fiber of $\sfF\arw\sfS_{\sfR}$ is $\Oc_{\PP^1}(-1)$, so its derived pushforward vanishes. Equation~\eqref{eq:tautological_transform_divisor_relation} and the projection formula now yield
\begin{equation*}
\Phi\bigl(\Oc(\zeta_{\sfX})\bigr) \simeq \Oc_{\sfR}(\zeta_{\sfR}) \otimes \mathbf{R}q_*\Oc_{\sfY}(\sfF) \simeq \Oc_{\sfR}(\zeta_{\sfR}).
\end{equation*}
Tensoring this calculation by $\Oc(mh)$ gives $\Phi(\Oc(mh+\zeta_{\sfX}))\simeq\Oc(mh+\zeta_{\sfR})$ for every $m\in\ZZ$.

Let
\begin{equation*}
\operatorname{ev}_{\sfX}: W^\vee\otimes\Oc_{\wt\sfX} \arw \Uc^\vee(-\sfE)
\end{equation*}
be the evaluation map in Equation~\eqref{eq:tautological_evaluation_complex}. By the identification $H^0(\wt\sfX,\Uc^\vee(-\sfE))=W^\vee$ in Lemma~\ref{lem:tautological_evaluation_mutation}, its composite with $\Uc^\vee(-\sfE)\hookrightarrow\Uc^\vee$ is induced by the inclusion $W^\vee\subset V^\vee$.

Write $r:\sfF\arw\sfZ$ for the natural projection, and denote by $\Oc_{\sfZ}(-1,0)$ the restriction of the tautological line bundle from the first factor of $\PP(P)\times\PP(W)$. A point $(\ell,w)\in\sfZ$ determines a Grassmannian line whose planes $U$ satisfy
\begin{equation*}
\ell\subset U\subset P\oplus w.
\end{equation*}
In particular, every plane on this line contains $\ell$, including the plane $P$ itself. The incidence description of $\sfS_{\sfX}$ therefore gives a line subbundle
\begin{equation*}
r^*\Oc_{\sfZ}(-1,0) \hookrightarrow (p^*\Uc)|_{\sfF}.
\end{equation*}
Dualizing and twisting, we obtain a surjection
\begin{equation*}
\bigl(p^*(\Uc^\vee(-\sfE))\bigr)|_{\sfF} \arw \Lc_{\sfF}, \qquad \Lc_{\sfF}:= r^*\Oc_{\sfZ}(1,0) \otimes\Oc_{\sfF}(-\sfE).
\end{equation*}
Together with restriction to $\sfF$, this defines a surjection
\begin{equation*}
\varepsilon: p^*(\Uc^\vee(-\sfE)) \arw \iota_*\Lc_{\sfF}.
\end{equation*}
Since $\sfF$ is an effective Cartier divisor and $\Lc_{\sfF}$ is a line bundle on $\sfF$, the kernel of $\varepsilon$ is locally free of rank two: this is the elementary-modification construction of \cite[\S~1, p.~570]{Ran2007NormalBundles}.

The composition
\begin{equation*}
W^\vee\otimes\Oc_{\sfY} \xrightarrow{p^*\operatorname{ev}_{\sfX}} p^*(\Uc^\vee(-\sfE)) \xrightarrow{\varepsilon} \iota_*\Lc_{\sfF}
\end{equation*}
is zero. Indeed, on $\sfF\setminus p^{-1}(\sfE)$ this follows from the fact that every linear form in $W^\vee$ annihilates $\ell\subset P$. This open subset is dense in every component of $\sfF$, since each flopping line meets $\sfE$ in one point. As $\sfF$ is reduced and $\Lc_{\sfF}$ is locally free, the vanishing extends over all of $\sfF$. Thus $p^*\operatorname{ev}_{\sfX}$ factors uniquely through the inclusion
\begin{equation*}
\ker(\varepsilon) \hookrightarrow p^*(\Uc^\vee(-\sfE)).
\end{equation*}

Let us now consider the dual tautological quotient from the Euler sequence on $\sfB$:
\begin{equation*}
\operatorname{ev}_{\sfB}: W^\vee\otimes\Oc_{\sfB}\arw\Vc^\vee.
\end{equation*}
Its pullback to $\sfY$ fits into the exact sequence
\begin{equation*}
0\arw\Oc_{\sfY}(-\zeta_{\sfR}) \arw W^\vee\otimes\Oc_{\sfY} \xrightarrow{q^*\psi^*\operatorname{ev}_{\sfB}} q^*\psi^*\Vc^\vee \arw0.
\end{equation*}
On the dense open set where the birational correspondence identifies a plane $U\subset V$ with its two-dimensional image in $W$, the first term is the annihilator of this image. Consequently, its composite with the map $W^\vee\otimes\Oc_\sfY\arw \ker(\varepsilon)$ coming from the unique factorization above, is zero on this open set, and hence everywhere, since both its source and target are locally free. We obtain a diagram as follows:
\begin{equation*}
\begin{tikzcd}[column sep = huge]
\Oc_\sfY(-\zeta_\sfR)\ar[hook]{r}\ar{d} & W^\vee\otimes\Oc_\sfY \ar[two heads]{r}{q^*\psi^*\operatorname{ev}_{\sfB}} \ar{d}\ar{dr}{p^*\operatorname{ev}_\sfX} & q^*\psi^*\Vc^\vee\ar[dashed]{d}{\vartheta} \\
0\ar{r} & \ker(\varepsilon)\ar[swap, hook]{r} & p^*(\Uc^\vee(-\sfE))
\end{tikzcd}
\end{equation*}
which proves the existence of the dashed arrow, i.e. a morphism
\begin{equation*}
\vartheta: q^*\psi^*\Vc^\vee \arw p^*(\Uc^\vee(-\sfE))
\end{equation*}
whose image is contained in $\ker(\varepsilon)$ and which satisfies
\begin{equation*}
p^*\operatorname{ev}_{\sfX} = \vartheta\circ q^*\psi^*\operatorname{ev}_{\sfB}.
\end{equation*}
The induced morphism from $q^*\psi^*\Vc^\vee$ to $\ker(\varepsilon)$ is generically an isomorphism.

We now show that this induced morphism is an isomorphism everywhere. Taking determinants of $0\arw\ker(\varepsilon)\arw p^*(\Uc^\vee(-\sfE))\arw \iota_*\Lc_{\sfF}\arw 0$ gives

\begin{equation*}
\begin{aligned}
\det\ker(\varepsilon) &\simeq \det\bigl(p^*(\Uc^\vee(-\sfE))\bigr) \otimes\Oc_{\sfY}(-\sfF)\\
&\simeq \Oc_{\sfY}(\zeta_{\sfX}-\sfF)\\
&\simeq \Oc_{\sfY}(\zeta_{\sfR}) \simeq \det(q^*\psi^*\Vc^\vee).
\end{aligned}
\end{equation*}
The determinant of the induced morphism $q^*\psi^*\Vc^\vee\arw \ker(\varepsilon)$ is therefore a nonzero global section of $\Oc_{\sfY}$. Since $\sfY$ is connected and projective, this section is a nonzero constant. It follows that $q^*\psi^*\Vc^\vee\arw\ker(\varepsilon)$ is an isomorphism. We have proved the exact sequence
\begin{equation}\label{eq:tautological_elementary_transform}
0\arw q^*\psi^*\Vc^\vee \xrightarrow{\vartheta}p^*(\Uc^\vee(-\sfE)) \xrightarrow{\varepsilon}\iota_*\Lc_{\sfF} \arw0.
\end{equation}

Let $\sfC$ be a fiber of $\sfF\arw\sfS_{\sfR}$. The map $r$ is constant on $\sfC$, while $p$ maps $\sfC$ isomorphically onto a flopping line. Lemma~\ref{lem:exceptional_divisor_and_canonical_class} gives $p^*\sfE\cdot\sfC=1$, so the explicit description of $\Lc_{\sfF}$ yields
\begin{equation*}
\Lc_{\sfF}|_{\sfC}\simeq\Oc_{\sfC}(-1).
\end{equation*}
Hence $\mathbf{R}q_*\iota_*\Lc_{\sfF}=0$ by cohomology and base change. Applying $\mathbf{R}q_*$ to Equation~\eqref{eq:tautological_elementary_transform} and using the projection formula gives an isomorphism
\begin{equation*}
\mathbf{R}q_*\vartheta: \psi^*\Vc^\vee \xrightarrow{\ \sim\ } \Phi\bigl(\Uc^\vee(-\sfE)\bigr).
\end{equation*}

To compute the transform of $\Ac_{\sfX}^\bullet$, we retain the compatibility with evaluation built into the construction of $\vartheta$. It gives the commutative diagram
\begin{equation*}
\begin{tikzcd}[column sep=huge,row sep=large]
W^\vee\otimes\Oc_{\sfY} \ar[r,two heads,"q^*\psi^*\operatorname{ev}_{\sfB}"] \ar[d,equal] & q^*\psi^*\Vc^\vee \ar[d,"\vartheta"] \\
W^\vee\otimes\Oc_{\sfY} \ar[r,"p^*\operatorname{ev}_{\sfX}"'] & p^*(\Uc^\vee(-\sfE)).
\end{tikzcd}
\end{equation*}
After applying $\mathbf{R}q_*$, the isomorphism $\mathbf{R}q_*\vartheta$ identifies $\Phi(\operatorname{ev}_{\sfX})$ with $\psi^*\operatorname{ev}_{\sfB}$. Consequently, Equation~\eqref{eq:tautological_evaluation_complex} and the exactness of $\Phi$ give
\begin{equation*}
\Phi(\Ac_{\sfX}^\bullet) \simeq \operatorname{Cone}\left( W^\vee\otimes\Oc_{\sfR} \xrightarrow{\psi^*\operatorname{ev}_{\sfB}} \psi^*\Vc^\vee \right).
\end{equation*}
The pullback of the dual tautological sequence to $\sfR$ is
\begin{equation*}
0\arw\Oc_{\sfR}(-\zeta_{\sfR}) \arw W^\vee\otimes\Oc_{\sfR} \xrightarrow{\psi^*\operatorname{ev}_{\sfB}} \psi^*\Vc^\vee \arw0.
\end{equation*}
Thus the preceding cone is $\Oc_{\sfR}(-\zeta_{\sfR})[1]$, as required.

\end{proof}

\begin{remark}\label{rem:negative_zeta_flop_transform}
The same argument does not identify all corresponding line bundles under the flop. For example,
\begin{equation*}
\Phi\bigl(\Oc(-\zeta_{\sfX})\bigr)= \Ic_{\sfS_{\sfR}}\otimes\Oc_{\sfR}(-\zeta_{\sfR}),
\end{equation*}
because $\mathbf{R}q_*\Oc_{\sfY}(-\sfF)=\Ic_{\sfS_{\sfR}}$. This is why it is necessary to transform the evaluation complex itself in Lemma~\ref{lem:tautological_flop_transform}.
\end{remark}

Let $\psi:\sfR\arw \sfB$ be a flat quadric fibration of relative dimension $m\geq1$ over a smooth projective base, and suppose that $\sfR$ is smooth. If $\Lc$ is its relative hyperplane line bundle and $\Cc_0$ its even Clifford algebra, Kuznetsov's theorem \cite[Theorem 4.2]{kuznetsov_quadric_fibrations_intersection_quadrics} gives a fully faithful functor from $\dbcoh(\sfB,\Cc_0)$ and a decomposition
\begin{equation}\label{eq:general_quadric_fibration_decomposition}
\dbcoh(\sfR)=\langle\dbcoh(\sfB,\Cc_0), \psi^*\dbcoh(\sfB)\otimes\Lc,\ldots, \psi^*\dbcoh(\sfB)\otimes\Lc^m\rangle.
\end{equation}
We write the Clifford category here by its image under this functor. In our setting flatness and smoothness were checked in Section~\ref{sec:birational_quadric_fibration}. Let $\Cc$ be the even Clifford algebra of $\sfR\arw\PP^2$, and let $\Bc$ denote the image of its Clifford category after tensoring Equation~\eqref{eq:general_quadric_fibration_decomposition} by $\Oc(-h)$. Thus
\begin{equation}\label{eq:gm_quadric_fibration_decomposition}
\dbcoh(\sfR)=\left\langle\Bc, \bigl(\psi^*\dbcoh(\PP^2)\otimes\Oc(kh)\bigr)_{k=0}^{r-1} \right\rangle, \qquad \Bc\simeq\dbcoh(\PP^2,\Cc).
\end{equation}

\begin{theorem}[Theorem~\ref{thm:intro_kuznetsov_clifford_equivalence}]\label{thm:kuznetsov_clifford_equivalence}
Let $\sfX$ be an $n$-dimensional ordinary GM variety, with $3\leq n\leq5$, satisfying Assumption~\ref{ass:generality}. Then
\begin{equation*}
\wt{\Kc u}(\sfX)\simeq\dbcoh(\PP^2,\Cc),
\end{equation*}
where $\Cc$ is the even Clifford algebra of the associated quadric fibration $\sfR\arw\PP^2$ of relative dimension $n-2$. In dimension four, if $\sfS\arw\PP^2$ is the discriminant double cover and $\alpha$ is its Clifford Brauer class, this gives
\begin{equation*}
\wt{\Kc u}(\sfX)\simeq\dbcoh(\PP^2,\Cc)\simeq\dbcoh(\sfS,\alpha),
\end{equation*}
with $\sfS$ a degree-two polarized $K3$ surface.
\end{theorem}
\begin{proof}
Apply $\Phi$ to Proposition~\ref{prop:mutated_gm_decomposition}. By Lemma~\ref{lem:tautological_flop_transform} and compatibility with twists by $h$, its $k$th exceptional block becomes
\begin{equation*}
\left\langle\Oc(-\zeta_{\sfR}+kh)[1],\Oc(kh),\Oc(\zeta_{\sfR}+kh)\right\rangle =\psi^*\dbcoh(\PP^2)\otimes\Oc(kh).
\end{equation*}
The equality is Beilinson's decomposition of $\PP^2$, and the shift of its first generator does not change the block. The exceptional complement is therefore exactly the complement in Equation~\eqref{eq:gm_quadric_fibration_decomposition}, with the same ordering and the same embeddings. Taking its right orthogonal gives
\begin{equation*}
\Phi\bigl(\RR_{\Oc(-\sfE)}\wt{\Kc u}(\sfX)\bigr)=\Bc.
\end{equation*}
Mutation restricts to an equivalence on the residual category, so this proves the first assertion.

For $n=4$, the quadratic form has rank four, the discriminant is a smooth sextic, and all singular fibers have corank one. The Clifford algebra becomes Azumaya over its center, whose spectrum is the smooth discriminant double cover $\tau:\sfS\arw\PP^2$ \cite{auel_parimala_suresh}. Morita equivalence identifies its category with $\dbcoh(\sfS,\alpha)$; see \cite[Theorem 5.10]{xie_quadric_bundles}. The double-cover canonical bundle formula and $\tau_*\Oc_{\sfS}=\Oc_{\PP^2}\oplus\Oc_{\PP^2}(-3)$ give $K_{\sfS}=0$ and $H^1(\sfS,\Oc_{\sfS})=0$. Finally the pullback of a line has self-intersection two and is ample, so $\sfS$ has the stated polarization.
\end{proof}

\begin{remark}\label{rem:smaller_spinor_resolution}
When $n$ is odd, the exceptional quadric $\sfE$ is even-dimensional. Choose one of its two normalized half-spinor bundles $\Sc_{\sfE}$, and denote the other by $\Sc'_{\sfE}$. The standard dual Lefschetz collection \cite[Theorem 2.35]{cattani_et_al} can be written
\begin{equation*}
\dbcoh(\sfE)=\langle\Oc_{\sfE}(-rh),\ldots,\Oc_{\sfE}(-2h), \Sc_{\sfE}(-h),\Oc_{\sfE}(-h),\Sc_{\sfE},\Oc_{\sfE}\rangle;
\end{equation*}
the initial string is empty when $n=3$. The associated categorical resolution \cite[Theorem 1]{kuznetsov_resolutions_of_singularities} has a \emph{different} residual category, which we denote by $\wt{\Kc u}_{\mathrm{sp}}(\sfX)$. It fits into
\begin{equation*}
\begin{split}
\dbcoh\bigl(\wt \sfX\bigr)=\langle &j_*\Oc_{\sfE}(-rh),\ldots,j_*\Oc_{\sfE}(-2h), j_*\Sc_{\sfE}(-h),j_*\Oc_{\sfE}(-h),\\
&\wt{\Kc u}_{\mathrm{sp}}(\sfX),\Oc,\Uc^\vee,\ldots, \Oc\bigl((r-1)\sfH\bigr),\Uc^\vee\bigl((r-1)\sfH\bigr)\rangle.
\end{split}
\end{equation*}
Right-mutate $j_*\Sc_{\sfE}(-h)$ through $j_*\Oc_{\sfE}(-h)$. Comparing with Equation~\eqref{eq:resolved_gm_semiorthogonal_decomposition} gives
\begin{equation*}
\wt{\Kc u}(\sfX)=\langle\Fc^\bullet,\wt{\Kc u}_{\mathrm{sp}}(\sfX)\rangle, \qquad \Fc^\bullet=\RR_{j_*\Oc_{\sfE}(-h)}j_*\Sc_{\sfE}(-h),
\end{equation*}
where $\Fc^\bullet$ is exceptional. More explicitly, the half-spinor exact sequence gives $\Fc^\bullet\simeq j_*\Sc'_{\sfE}[-1]$; see \cite[Lemma 3.4]{cattani_et_al}. Transporting this decomposition by the theorem yields a corresponding decomposition of $\dbcoh(\PP^2,\Cc)$. In particular, the larger resolution identified with the Clifford category and the smaller spinor resolution must be distinguished.
\end{remark}

\section{Nodal GM varieties and hyperbolic reduction}\label{sec:gm_hyperbolic_reduction}

Kuznetsov and Perry proved, using categorical cones, that a smooth Gushel--Mukai variety of dimension five or six has the same Kuznetsov component as a suitable smooth Gushel--Mukai variety of dimension three or four, respectively \cite[Theorem 1.6]{kuznetsov_perry_categorical_cones}. In this section we study a related construction for nodal varieties. The quadratic bundles of Section~\ref{sec:birational_quadric_fibration} admit useful enlargements by hyperbolic planes. These give the comparison between dimensions five and three, as well as the relationship with Verra varieties. We also construct a family of rational nodal fourfolds whose Clifford categories are untwisted derived categories of degree-two $K3$ surfaces.

\subsection{Hyperbolic reduction in families}\label{subsec:hyperbolic_reduction}
Let $\sfS$ be a smooth variety, let $\Vc$ be a vector bundle, and let
\begin{equation*}
q:\Vc\arw \Vc^\vee
\end{equation*}
be a symmetric morphism. We use the convention that $\PP_{\sfS}(\Vc)$ parametrizes lines in $\Vc$; thus $q$ defines a section of $\Oc_{\PP_{\sfS}(\Vc)}(2)$. An isotropic subbundle $\Fc\subset \Vc$ is called \emph{regular} if the composition $\Vc\xrightarrow{q}\Vc^\vee\arw \Fc^\vee$ is surjective. In this case
\begin{equation*}
\Fc^\perp=\ker(\Vc\arw \Fc^\vee),\qquad \Vc_{\Fc}=\Fc^\perp/\Fc
\end{equation*}
are vector bundles, and the form induced on $\Vc_{\Fc}$ defines the \emph{hyperbolic reduction}. Its rank is $\operatorname{rk}(\Vc)-2\operatorname{rk}(\Fc)$.

For example, if $\sfQ\subset\PP^r$ is smooth at $x=[1:0:\cdots:0]$, coordinates can be chosen so that its equation is
\begin{equation*}
x_0x_1+q'(x_2,\ldots,x_r)=0.
\end{equation*}
The tangent section $\sfQ\cap\{x_1=0\}$ is a cone with vertex $x$, and its projection from $x$ is the quadric $q'=0$ in $\PP^{r-2}$. This is the reduction along the isotropic line corresponding to $x$. Projection of the whole quadric from $x$, on the other hand, is birational onto $\PP^{r-1}$.

Locally on $\sfS$, regularity permits a decomposition
\begin{equation*}
(\Vc,q)\simeq \Hc(\Fc)\perp(\Vc_{\Fc},q_{\Fc}),
\end{equation*}
where $\Hc(\Fc)=\Fc\oplus \Fc^\vee$ has its evaluation pairing. In particular, reduction preserves the discriminant and the coranks of the fibers. For forms with simple degeneration along a smooth divisor, we will use the Morita invariance of even Clifford algebras under this operation \cite[Theorem 1.8.7]{auel_bernardara_bolognesi_quadric_fibrations}. All reductions below have regular centers; when smoothness of a discriminant is needed, it will be stated explicitly.

Put $\sfB=\PP\bigl(\wedge^2 W\bigr)$ and write $\Oc(a)=\Oc_{\sfB}(a\zeta)$ in this section. Here we retain the dimension subscripts on the quadratic bundles $\Ec_n$ and the even Clifford algebras $\Cc_n$ when comparing different dimensions. The relevant quadratic forms are forms on
\begin{equation*}
\Kc_5= P\otimes\Tc_{\sfB}^\vee(1)\oplus\Oc(-1), \qquad \Kc_n=\Ec_n^\vee\quad (n=3,4).
\end{equation*}
The linear equations defining $\sfX$ induce exact sequences
\begin{equation}\label{eq:quadratic_bundle_linear_sections}
\begin{gathered}
0\arw \Kc_n\arw \Kc_5\arw \Oc^{\oplus(5-n)}\arw0,\\
0\arw\Oc^{\oplus(5-n)}\arw \Ec_5\arw\Ec_n\arw0.
\end{gathered}
\end{equation}
Thus $\sfR\subset\PP_{\sfB}(\Kc_n)$ is defined by a section of $\Sym^2\Ec_n$. The subbundles of $\Ec_5$ in the second sequence specify linear equations, rather than isotropic centers for this form.

\begin{remark}\label{rem:linear_sections_and_reductions}
There is a useful obstruction to identifying a prescribed codimension-two section with a direct regular reduction. If a regular isotropic line $\Fc\subset \Kc_5$ had reduced bundle $\Kc_3$, then
\begin{equation*}
[\Kc_5]=[\Kc_3]+[\Fc]+[\Fc^\vee]
\end{equation*}
in $K_0(\sfB)$. Equation~\eqref{eq:quadratic_bundle_linear_sections} would give $[\Fc]+[\Fc^\vee]=2[\Oc]$. Writing $\Fc=\Oc(a)$ and comparing Chern characters in degree four gives $a^2\zeta^2=0$, hence $a=0$. But $H^0(\sfB,\Kc_5)=0$, so $\Oc$ does not embed in $\Kc_5$. The comparison below therefore uses a common enlargement of the quadratic bundles; it does not identify a fixed linear section with a direct reduction.
\end{remark}

\begin{lemma}\label{lem:split_quadratic_presentation}
For general linear equations defining $\sfX$, where $3\leq n\leq5$, there is an exact sequence
\begin{equation}\label{eq:gm_quadratic_bundle_resolution}
\begin{gathered}
0\arw \Lc\xrightarrow{d}\Ac_n\arw\Ec_n\arw0,\\
\Lc=\Oc(-1)^{\oplus2},\qquad \Ac_n=\Oc^{\oplus(n+1)}\oplus\Oc(1).
\end{gathered}
\end{equation}
Every section $q\in H^0(\sfB,\Sym^2\Ec_n)$ lifts to a section $\widehat q\in H^0(\sfB,\Sym^2\Ac_n)$. The form
\begin{equation}\label{eq:gm_hyperbolic_enlargement}
Q_n=\begin{pmatrix}\widehat q&d\\d^\vee&0\end{pmatrix} \quad\text{on}\quad \Ac_n^\vee\oplus \Lc=\Oc^{\oplus(n+1)}\oplus\Oc(-1)^{\oplus3}
\end{equation}
has the regular isotropic subbundle $\Lc$, and its reduction along $\Lc$ is $(\Kc_n,q)$. Moreover, ambient quadratic equations singular at $x$ map onto $H^0(\sfB,\Sym^2\Ec_n)$.
\end{lemma}
\begin{proof}
The two Euler sequences give
\begin{equation*}
0\arw\Oc(-1)^{\oplus2}\arw \Oc^{\oplus6}\oplus\Oc(1)\arw\Ec_5\arw0.
\end{equation*}
The $5-n$ sections defining the linear section lift to the middle term. Their constant components are independent for a general choice. After a change of basis, and subtracting their $\Oc(1)$ components, their quotient is $\Ac_n$, proving~Equation~\eqref{eq:gm_quadratic_bundle_resolution}.

The symmetric-square resolution is
\begin{equation*}
0\arw\wedge^2\Lc\arw \Lc\otimes \Ac_n \arw\Sym^2\Ac_n\arw\Sym^2\Ec_n\arw0.
\end{equation*}
Here $H^2\bigl(\sfB,\wedge^2\Lc\bigr)=H^2\bigl(\sfB,\Oc(-2)\bigr)=0$ and $H^1(\sfB,\Lc\otimes \Ac_n)=0$. Splitting this sequence into two short exact sequences therefore proves surjectivity on global sections. In~Equation~\eqref{eq:gm_hyperbolic_enlargement}, the pairing with $\Lc$ is the surjection $d^\vee:\Ac_n^\vee\arw \Lc^\vee$. Thus $\Lc^\perp=\Kc_n\oplus \Lc$, and its quotient by $\Lc$ has the original form. This is the explicit lifting construction of \cite[Proposition 3.6]{bini_kapustka2}.

To verify the last assertion, products of global sections generate $H^0(\sfB,\Sym^2\Ec_5)$. This follows directly from the Euler sequence: the multiplication maps onto $H^0\bigl(\Sym^2(\Tc_{\sfB}(-1))\bigr)$, $H^0\bigl(\Tc_{\sfB}(-1)\otimes \Tc_{\sfB}(-1)\bigr)$, and $H^0(\Tc_{\sfB})$ are surjective, as is the usual multiplication onto $H^0\bigl(\Oc(2)\bigr)$. These are the summands of the symmetric square in question. Passing from $\Ec_5$ to $\Ec_n$ preserves surjectivity on global sections of symmetric squares, using~Equation~\eqref{eq:quadratic_bundle_linear_sections}, $H^1(\sfB,\Ec_5)=0$, and $H^2(\sfB,\Oc)=0$. The global sections of $\Ec_5$ are precisely the linear coordinates after projection from $x$, so their products are ambient quadrics singular at $x$.
\end{proof}

Here is a family for which an actual isotropic line does occur. Fix a general hyperplane section $\sfH=\sfG\cap\VV(s)$ through $x$. The alternating form defining $s$ has rank four. We may choose the complement $W$ of $P$ to be isotropic for that form: its image in the four-dimensional symplectic quotient is a Lagrangian complement to the image of $P$, and it contains the one-dimensional kernel of the alternating form. For this choice, $s$ has no $\wedge^2 W$ coefficient. In particular,
\begin{equation}\label{eq:fourfold_quadratic_bundle_splitting}
\Kc_4=\Kc'\oplus\Oc(-1),\qquad 0\arw \Kc'\arw P\otimes\Tc_{\sfB}^\vee(1) \xrightarrow{s}\Oc\arw0.
\end{equation}
The last map is surjective, and $(\Kc')^\vee$ is globally generated. The summand $\Fc=\Oc(-1)$ parametrizes the plane $\sfP=\operatorname{Gr}(2, W)\subset \sfH$, which is disjoint from $x$.

\begin{proposition}\label{prop:fixed_plane_rational_fourfolds}
Fix $\sfH$, $W$, and $\sfP$ as above. In the vector space of ambient quadrics singular at $x$, the condition of containing $\sfP$ has codimension six. A general quadric satisfying this condition defines a one-nodal ordinary GM fourfold $\sfX_4$ for which the following hold:
\begin{enumerate}
\item $\sfR_4\arw \sfB$ has a regular isotropic section, and $\sfX_4$ is rational;
\item the hyperbolic reduction is a smooth double cover $\sfS\arw \sfB$ branched along a smooth sextic, so $\sfS$ is a $K3$ surface with an ample class of square two;
\item $\dbcoh(\sfB,\Cc_4)\simeq\dbcoh(\sfS)$.
\end{enumerate}
In particular, $\wt{\Kc u}(\sfX_4)\simeq\dbcoh(\sfS)$.
\end{proposition}
\begin{proof}
Restriction of ambient quadrics singular at $x$ to $\sfP\simeq\PP^2$ is onto $H^0\bigl(\sfP,\Oc_{\sfP}(2)\bigr)$. Its vanishing is therefore six independent linear conditions. On the quadratic bundle this is exactly $q|_{\Fc}=0$. By~Equation~\eqref{eq:fourfold_quadratic_bundle_splitting}, the form has the shape
\begin{equation}\label{eq:fourfold_isotropic_form}
q=\begin{pmatrix}a&b^\vee\\ b&0\end{pmatrix}, \qquad a\in H^0\bigl(\sfB,\Sym^2(\Kc')^\vee\bigr),\quad b:\Kc'\arw\Oc(1).
\end{equation}
The bundle $(\Kc')^\vee(1)$ is globally generated and has rank three. A general section is nowhere zero on the surface $\sfB$. Hence $b$ is surjective, $\Fc$ is regular isotropic, and the reduction is the form induced by $a$ on
\begin{equation*}
\Dc=\ker(b:\Kc'\arw\Oc(1)).
\end{equation*}
Since $(\Kc')^\vee$ is globally generated, the forms obtained by restricting $a$ generate every fiber of $\Sym^2\Dc^\vee$. Generic transversality to the rank strata of symmetric $2\times2$ matrices now gives a smooth rank-one locus and an empty rank-zero locus. Here the latter stratum has codimension three, which exceeds $\dim \sfB$. Thus the reduced quadric is a finite flat double cover with smooth discriminant.

We have $\det \Kc'=\Oc(-2)$ and $\det \Dc=\Oc(-3)$. The determinant of the reduced form is consequently a section of $\Oc(6)$. Its double cover $\tau:\sfS\arw \sfB$ satisfies
\begin{equation*}
\tau_*\Oc_{\sfS}=\Oc_{\sfB}\oplus\Oc_{\sfB}(-3),\qquad \omega_{\sfS}=\tau^*(\omega_{\sfB}\otimes\Oc_{\sfB}(3))=\Oc_{\sfS}.
\end{equation*}
It follows that $H^1(\sfS,\Oc_{\sfS})=0$, so $\sfS$ is a $K3$ surface, and $(\tau^*\zeta)^2=2$. The even Clifford algebra of a rank-two form is the algebra of its discriminant double cover. Applying hyperbolic Morita invariance gives the categorical equivalence in (3), and Theorem~\ref{thm:kuznetsov_clifford_equivalence} identifies this with $\wt{\Kc u}(\sfX_4)$. Finally, the regular section gives a rational point on the generic quadric surface. Projection from it makes $\sfR_4$ rational over $\CC(\sfB)$, and hence $\sfX_4$ is rational.

For completeness, the required geometric open conditions are nonempty in this linear system. Its base locus on the smooth fivefold $\sfH$ is $\sfP\cup\{x\}$. Away from this locus Bertini applies. Along $\sfP$, the mixed quadratic coefficients generate the rank-three twisted conormal bundle, so a general differential has no zero. At $x$, the quadratic part on $T_x\sfH$ is unrestricted and can be chosen nondegenerate. The same unrestricted quadratic form on $P\otimes W$ cuts the Segre variety defining $\sfZ_4$ transversely. The ambient quadric can also have vertex exactly $x$: in the decomposition $(P\otimes W)\oplus\wedge^2 W$, a general matrix with lower diagonal block zero is nondegenerate, as follows by taking its Schur complement. Finally, the restriction map from ambient quadrics to the forms in~Equation~\eqref{eq:fourfold_isotropic_form} is onto, by Lemma~\ref{lem:split_quadratic_presentation}. Thus the smooth-discriminant condition proved above is a nonempty open condition in the same linear system. These finitely many nonempty open subsets intersect.
\end{proof}

\begin{remark}\label{rem:degree_two_discriminant_surface}
The surface $\sfS$ in Proposition~\ref{prop:fixed_plane_rational_fourfolds} is the discriminant double cover of the reduced quadratic bundle. No identification with the resolution of a prescribed nodal linear-section GM surface is needed or claimed. In particular, the construction proves the untwisted degree-two $K3$ description directly from a regular isotropic section.
\end{remark}

Allowing the plane to vary gives a larger explicit family.
\begin{corollary}\label{cor:moving_plane_codimension}
Fix a general hyperplane section $\sfH$ through $x$. In the vector space of ambient quadratic equations singular at $x$, the closure of the locus of quadrics containing $\operatorname{Gr}(2,W)$ for some isotropic complement $W$ of $P$ has codimension three. A general member defines a rational one-nodal GM fourfold whose resolved Kuznetsov component $\wt{\Kc u}(\sfX_4)$ is equivalent to $\dbcoh(\sfS)$ for a degree-two polarized $K3$ surface $\sfS$.
\end{corollary}
\begin{proof}
Choose a basis $e_1,\ldots,e_5$ of $V$ such that $P=\langle e_1,e_2\rangle$ and the hyperplane equation is $s=x_{13}+x_{24}$ in Pl\"ucker coordinates. Its isotropic complements of $P$ are
\begin{equation*}
W(a,b,c)=\langle e_3+ae_1+be_2,\ e_4+be_1+ce_2,\ e_5\rangle,\qquad (a,b,c)\in\mathbb A^3.
\end{equation*}
Indeed, every isotropic three-space contains the radical $\langle e_5\rangle$, and the remaining condition is that its image be a Lagrangian complement in the symplectic quotient. The space of ambient quadrics singular at $x$ has dimension $45$. Containment of each plane imposes six independent linear conditions, so the incidence of such complements and quadrics is a rank-$39$ vector bundle over $\mathbb A^3$, of dimension $42$.

We check that its projection to the space of quadrics is generically finite onto its image. At $(a,b,c)=(0,0,0)$, consider
\begin{equation*}
q_0=x_{14}x_{34}+x_{15}x_{35}+x_{25}x_{45}.
\end{equation*}
Write $w_3=e_3+ae_1+be_2$ and $w_4=e_4+be_1+ce_2$. Substitution of $Xw_3\wedge w_4+Yw_3\wedge e_5+Zw_4\wedge e_5$ into $q_0$ gives the first variation
\begin{equation*}
a(X^2+Y^2)+2bYZ+cZ^2.
\end{equation*}
The three coefficient quadrics are linearly independent. Thus the differential of the incidence projection has zero kernel at this point. This remains a nonempty open condition in the fixed-plane linear system and intersects the nonempty open of geometrically good quadrics from Proposition~\ref{prop:fixed_plane_rational_fourfolds}. The image closure consequently has dimension $42$ and codimension three. A general point of it comes from that good incidence locus, so the remaining assertions follow from the same proposition.
\end{proof}

\subsection{The relationship with Verra threefolds and fourfolds}\label{subsec:gm_verra_comparison}
Recall that a Verra threefold is a smooth divisor of bidegree $(2,2)$ in $\PP^2\times\PP^2$ \cite[Section 1.1.6]{iliev1997theta}, and a Verra fourfold is a smooth double cover of $\PP^2\times\PP^2$ branched along such a divisor \cite[Section 0.3]{iliev_kapustka_kapustka_ranestad_verra}. Fix one projection $f:\sfV\arw \sfB$, where $n=\dim\sfV$. It is a conic fibration for $n=3$ and a quadric-surface fibration for $n=4$. We denote its Clifford component by $\Kc u(\sfV,f)$, keeping the chosen base action in the notation. After tensoring the standard quadric-fibration decomposition by the inverse relative hyperplane bundle, one may write, for example,
\begin{equation*}
\dbcoh(\sfV_3)=\langle\Kc u(\sfV_3,f),f^*\Oc_{\sfB}(-1),f^*\Oc_{\sfB},f^*\Oc_{\sfB}(1)\rangle.
\end{equation*}
The equivalences below respect this chosen projection.

\begin{theorem}[Theorem~\ref{thm:intro_gm_verra_comparison}]\label{thm:gm_verra_hyperbolic_comparison}
Let $\sfX$ be a general ordinary one-nodal GM variety, with $3\leq n\leq5$. Set $m=3$ for odd $n$ and $m=4$ for even $n$. There are a smooth Verra variety $\sfV_m$, a projection $f:\sfV_m\arw \sfB$, and a quadric fibration $\sfW\arw \sfB$ such that both $\sfR$ and $\sfV_m$ are regular hyperbolic reductions of $\sfW$. In particular,
\begin{equation*}
\dbcoh(\sfB,\Cc_n)\simeq\Kc u(\sfV_m,f).
\end{equation*}
For $n=3,4$, the varieties $\sfR$ and $\sfV_m$ are birational over $\sfB$.
\end{theorem}
\begin{proof}
Use the enlargement in Equation~\eqref{eq:gm_hyperbolic_enlargement}. Relative to $\Oc^{\oplus(n+1)}\oplus\Oc(-1)^{\oplus3}$, write its matrix as
\begin{equation*}
Q_n=\begin{pmatrix}A_0&B_1\\B_1^\vee&C_2\end{pmatrix}.
\end{equation*}
The entries of $A_0$ are constant, those of $B_1$ are linear, and those of $C_2$ are quadratic. For a general original form $q$, the matrix $A_0$ is invertible. Indeed, any constant symmetric matrix occurs in $\widehat q$, and changing a lift does not change this block: the global sections of $\Lc\otimes \Ac_n$ involve the $\Oc(1)$ summand only.

Completing the square is a global bundle automorphism, and gives the orthogonal decomposition
\begin{equation}\label{eq:verra_schur_complement}
Q_n\simeq A_0\perp\bigl(C_2-B_1^\vee A_0^{-1}B_1\bigr).
\end{equation}
Over $\CC$, an invertible constant form of even rank is a sum of hyperbolic planes. Removing it for $n=3,5$ leaves a rank-three form on $\Oc(-1)^{\oplus3}$. Its equation is quadratic in the fiber coordinates and quadratic in the base coordinates, so it defines a $(2,2)$ divisor in $\PP^2\times\PP^2$.

For $n=4$, write the constant rank-five form as two hyperbolic planes and a one-dimensional form, normalized to $z^2$. The reduction is
\begin{equation*}
z^2+c(x;y)=0\quad\text{in}\quad\PP_{\sfB}\bigl(\Oc\oplus\Oc(-1)^{\oplus3}\bigr),
\end{equation*}
where $c$ has bidegree $(2,2)$. The locus with all three $y$ coordinates zero does not meet this quadric. Projection to $\PP_{\sfB}\bigl(\Oc(-1)^{\oplus3}\bigr)$ is therefore a finite double cover, branched along $c=0$, and gives a Verra fourfold.

We also check smoothness. The bundle $\Ec_n$ is globally generated, hence its symmetric square is generated by the forms under consideration. For a general form the map to symmetric matrices is transverse to the corank-one stratum and avoids corank at least two, whose codimension is three. Thus $\sfR$ has smooth total space and smooth discriminant. The regular reduction along $\Lc$, and the orthogonal decomposition in Equation~\eqref{eq:verra_schur_complement}, preserve these properties. For $n=4$, coranks and discriminants of $1\perp c$ are those of $c$, so its $(2,2)$ branch divisor is smooth as well. These are nonempty open conditions, compatible with the geometric generality assumptions of Section~\ref{sec:birational_quadric_fibration}.

Both centers used in the construction are regular. The asserted equivalence follows from hyperbolic Morita invariance. For $n=3,4$, the two reduced forms have the same rank and become isometric over $\CC(\sfB)$ by Witt cancellation, proving the last assertion.
\end{proof}

\begin{lemma}\label{lem:verra_dominance}
The constructions of Theorem~\ref{thm:gm_verra_hyperbolic_comparison} for $3\leq n\leq5$ dominate the family of Verra varieties of dimension $m$, where $m=3$ for odd $n$ and $m=4$ for even $n$, with a chosen projection. Consequently, a general nodal GM variety of dimension $n$ can be associated with a Verra variety in any prescribed nonempty open subset of the corresponding family.
\end{lemma}
\begin{proof}
We give the parameter calculation, including the reverse construction. Let $C$ be a symmetric $3\times3$ matrix of quadratic forms in $x,y,z$, defining a general $(2,2)$ divisor, hence a Verra threefold or the branch divisor of a Verra fourfold. Put $r=n+1$, and fix the identity form $J$ on $\Oc^{\oplus r}$; here $r=4,5$, or $6$. We seek a matrix $N$ of size $r\times2$, with linear entries, such that
\begin{equation}\label{eq:verra_gram_equation}
N^\vee JN=-C_{[1,2],[1,2]}.
\end{equation}
The polynomial map $N\mapsto N^\vee JN$ is dominant. To see this for $r=4$, take
\begin{equation*}
N_0=\begin{pmatrix} x&0\\
y&x\\
z&y\\
0&z
\end{pmatrix}.
\end{equation*}
The map
\begin{equation*}
H^0\bigl(\Oc(1)\bigr)^{\oplus4}\arw H^0\bigl(\Oc(2)\bigr)^{\oplus2}, \qquad t\longmapsto N_0^\vee t
\end{equation*}
is an isomorphism. Indeed, if $xt_1+yt_2+zt_3=0$, the linear Koszul syzygies give $t_1=ay+bz$, $t_2=-ax+cz$, and $t_3=-bx-cy$ for constants $a,b,c$. The second equation $xt_2+yt_3+zt_4=0$ has coefficients $-a$, $-b$, and $-c$ at $x^2$, $xy$, and $y^2$, respectively. Thus $a=b=c=0$, and then $t_4=0$. Injectivity proves the assertion since both spaces have dimension twelve. Applying this surjectivity to each column shows that the derivative $\dot N\mapsto N_0^\vee\dot N+\dot N^\vee N_0$ is onto all symmetric $2\times2$ matrices of quadrics. For $r=5$ or $6$, append one or two zero rows; the same derivative remains onto. Thus a general right-hand side in~Equation~\eqref{eq:verra_gram_equation} has a solution.

The solution can be chosen with $N$ of rank two at every point of $\sfB$. This is an open condition and it holds for $N_0$: otherwise evaluation at a point of rank drop would contradict the surjectivity of the preceding multiplication map. We may also require the $r\times6$ coefficient matrix of $N$ to have rank $r$. These are nonempty open conditions on the space of $N$; dominance persists after imposing them.

Now consider $J\perp C$ on $\Oc^{\oplus r}\oplus\Oc(-1)^{\oplus3}$. Equation~\eqref{eq:verra_gram_equation} says that the graph
\begin{equation*}
\Oc(-1)^{\oplus2}\arw \Oc^{\oplus r}\oplus\Oc(-1)^{\oplus3}, \qquad u\longmapsto(Nu,u,0)
\end{equation*}
is isotropic. It is regular, since pairing with the constant summand is the everywhere surjective map $N^\vee J$. Subtracting $Nu$ from the first coordinates identifies the reduced bundle with the kernel of
\begin{equation*}
\Oc^{\oplus r}\oplus\Oc(-1)\arw\Oc(1)^{\oplus2}, \qquad (v,w)\longmapsto N^\vee Jv+C_{[1,2],3}w.
\end{equation*}
Its dual therefore has a presentation of the form in Equation~\eqref{eq:gm_quadratic_bundle_resolution}.

These presentations come from GM linear sections. Indeed, a general $r\times2$ matrix of linear forms is obtained from the two Euler columns by a constant quotient $\mathbb C^6\arw\mathbb C^r$. For $r=6$ it is an isomorphism; for $r=4,5$ its kernel supplies the $6-r=5-n$ hyperplane equations. The remaining row consists of two quadrics. Each is a sum of products of linear forms with $x,y,z$, so this row is obtained by a change of the $\Oc(1)$ component before taking the quotient. Finally, the last assertion of Lemma~\ref{lem:split_quadratic_presentation} lifts the reduced form to an ambient quadratic equation singular at $x$.

More explicitly, the parameter space used here consists of the full-rank matrices $N$ just described, two arbitrary quadratic entries $C_{13},C_{23}$, and an arbitrary quadratic entry $C_{33}$; its upper $2\times2$ block is defined by~Equation~\eqref{eq:verra_gram_equation}. This is an open subset of an affine space and is irreducible. The differential calculation proves that its map to the space of $C$ is dominant. Conversely, starting with a general GM presentation and a lift with invertible constant block, normalize that block to $J$ and complete the square. The original isotropic summand $\Lc$ becomes a graph $(N,I_2,0)$, and the remaining entries of the Schur complement are precisely the free coefficients $C_{13},C_{23},C_{33}$. Thus these parameters cover, up to changes of bundle coordinates, the open set of GM quadratic data with invertible constant block. The conditions that $\sfX$ have just the prescribed node and that Assumption~\ref{ass:generality} hold therefore cut out a nonempty open subset of this same parameter space. Restricting a dominant map from an irreducible variety to a nonempty open subset preserves dominance. This proves both assertions.
\end{proof}

\begin{theorem}[Theorem~\ref{thm:intro_gm_duality}]\label{thm:odd_gm_duality}
Let $\sfX_5$ be a general ordinary Gushel--Mukai fivefold with one node. There exists an ordinary Gushel--Mukai threefold $\sfX_3$ with one node such that
\begin{equation*}
\dbcoh(\sfB,\Cc_5)\simeq\dbcoh(\sfB,\Cc_3).
\end{equation*}
Consequently, $\wt{\Kc u}(\sfX_5)\simeq\wt{\Kc u}(\sfX_3)$. The threefold is obtained by the construction above; it is not asserted to be a prescribed linear section of $\sfX_5$.
\end{theorem}
\begin{proof}
Theorem~\ref{thm:gm_verra_hyperbolic_comparison} associates to $\sfX_5$ a Verra threefold $\sfV_3$ with the same Clifford category. The dominance assertion of Lemma~\ref{lem:verra_dominance} shows that $\sfV_3$ can be taken in the open subset of Verra threefolds obtained from one-nodal GM threefolds by the same construction. Composing the two equivalences gives the result. The categorical assertion follows from Theorem~\ref{thm:kuznetsov_clifford_equivalence}.
\end{proof}

\begin{corollary}[Theorem~\ref{thm:intro_gm_verra_comparison}]\label{cor:gm_verra_clifford_equivalence}
Let $\sfX$ be a general ordinary one-nodal Gushel--Mukai variety of dimension three or four. There is a smooth Verra variety $\sfV$ of the same dimension, with a chosen projection $f$, such that
\begin{equation*}
\dbcoh(\sfB,\Cc_n)\simeq\Kc u(\sfV,f).
\end{equation*}
In particular, $\wt{\Kc u}(\sfX)\simeq\Kc u(\sfV,f)$.
\end{corollary}
\begin{proof}
Apply Theorem~\ref{thm:gm_verra_hyperbolic_comparison} for $n=3,4$ and Theorem~\ref{thm:kuznetsov_clifford_equivalence}. Lemma~\ref{lem:verra_dominance} also permits the Verra variety to be chosen general in either dimension. The fourfold construction agrees with the split-resolution model in \cite[Section 4.1]{bini_kapustka2}.
\end{proof}

\begin{remark}\label{rem:verra_rationality_analogy}
A general member of the rational family in Proposition~\ref{prop:fixed_plane_rational_fourfolds} has an associated rational Verra fourfold whose Clifford component is equivalent to $\dbcoh(\sfS)$. Indeed, in the presentation obtained from~Equation~\eqref{eq:fourfold_quadratic_bundle_splitting}, the condition $q|_{\Fc}=0$ only sets the bottom quadratic coefficient of $\widehat q$ to zero; its constant $5\times5$ block is unrestricted and hence generally invertible. The smoothness conditions were verified within this family in the proof of the proposition. Theorem~\ref{thm:gm_verra_hyperbolic_comparison} therefore applies to a general member of this same family, and birationality transfers rationality to its Verra model.
\end{remark}

\subsection{Divisorial families of rational nodal fourfolds}\label{subsec:divisorial_rational_fourfolds}

Fix the node $x=[\wedge^2P]$ and a general hyperplane section $\sfH$ of $\operatorname{Gr}(2,V)$ through $x$. Let $\mathscr Q=\operatorname{Sym}^2A^\vee$ be the vector space of ambient quadratic equations singular at $x$, and let $\mathscr U\subset\mathscr Q$ be the nonempty Zariski-open subset satisfying Assumption~\ref{ass:generality}. We write $\sfX_q$ for the corresponding nodal fourfold. Its quadric-surface fibration has a smooth discriminant double cover $\sfS_q\arw\PP^2$, with polarization $h_q$ of square two and Clifford Brauer class $\alpha_q$. These double covers form a smooth projective family $p:\mathscr S\arw\mathscr U$.

\begin{lemma}\label{lem:dominant_discriminant_periods}
The family $(\sfS_q,h_q)$ dominates the moduli space of degree-two polarized $K3$ surfaces. In particular, its local period maps have rank nineteen on a nonempty Zariski-open subset of $\mathscr U$.
\end{lemma}
\begin{proof}
By Lemma~\ref{lem:verra_dominance}, allowing the hyperplane through $x$ to vary gives a dominant family of Verra fourfolds with a chosen projection. Their branch equations are symmetric $3\times3$ matrices $C$ of quadratic forms on $\PP^2$, and their discriminant sextics are $\det C=0$. These determinants dominate the space of plane sextics. Indeed, for a general pair consisting of a smooth plane sextic $D$ and a nonzero two-torsion line bundle $\eta$ on $D$, Beauville's symmetric resolution is
\begin{equation*}
0\arw\Oc_{\PP^2}(-4)^{\oplus3}\xrightarrow{\ C\ }\Oc_{\PP^2}(-2)^{\oplus3}\arw\eta\arw0,
\end{equation*}
and $\det C$ is an equation of $D$; see \cite[Proposition 4.6]{beauville_determinantal_hypersurfaces}. The parameter space of such pairs is a finite cover of the space of smooth sextics, so this gives the asserted dominance. Double covers branched along smooth sextics form a dense open subset of the moduli space of degree-two polarized $K3$ surfaces.

The hyperplane can be fixed. Its defining alternating form has rank four, vanishes on $\wedge^2P$, and has radical disjoint from $P$. The stabilizer of $P$ in $\operatorname{GL}(V)$ acts transitively on these forms up to scalar: after quotienting by the radical, choose a symplectic basis extending the image of $P$. Coordinate changes from this stabilizer fix $x$ and preserve the isomorphism class of the polarized discriminant double cover. Consequently the image in $K3$ moduli is the same for every hyperplane in this open orbit. Dominance therefore holds for the fixed general $\sfH$. Generic smoothness and local Torelli, \cite[Section 6.2]{huybrechts_k3_lectures}, give the assertion about the local period maps.
\end{proof}

We use the following form of the Noether--Lefschetz argument. Its congruence condition keeps track of the Clifford class in a family.

\begin{lemma}\label{lem:brauer_vanishing_divisors}
Let $p:\mathscr S\arw T$ be a smooth projective family of degree-two polarized complex $K3$ surfaces over a smooth irreducible quasi-projective variety, and suppose that its local period maps generically have rank nineteen. Let $\mathscr A$ be an Azumaya algebra of degree two on $\mathscr S$, and write $\alpha_t\in\operatorname{Br}(\mathscr S_t)[2]$ for its restriction. There is a countable collection of closed irreducible divisors $D_i\subset T$, whose union is dense in the complex topology, such that $\alpha_t=0$ for every $t\in D_i$.
\end{lemma}
\begin{proof}
The $\operatorname{PGL}_2$-torsor defined by $\mathscr A$ has an obstruction class $\beta\in H^2_{\mathrm{\acute et}}(\mathscr S,\mu_2)$ for the central extension
\begin{equation*}
1\arw\mu_2\arw\operatorname{SL}_2\arw\operatorname{PGL}_2\arw1.
\end{equation*}
Its image in the Brauer group is the class of $\mathscr A$. By comparison with singular cohomology and restriction to the fibers, $\beta$ gives a flat section $\bar b$ of $\mathbf{R}^2p_*\ZZ/2\ZZ$. On a simply connected analytic neighborhood in $T$, mark $\mathbf{R}^2p_*\ZZ$ by the $K3$ lattice $\Lambda$ and denote the polarization by $h\in\Lambda$, with $h^2=2$. Since $H^3(\mathscr S_t,\ZZ)=0$, choose an integral lift $b\in\Lambda$ of $\bar b$. The Kummer sequence gives
\begin{equation*}
0\arw\operatorname{Pic}(\mathscr S_t)/2\operatorname{Pic}(\mathscr S_t)\arw H^2(\mathscr S_t,\mu_2)\arw\operatorname{Br}(\mathscr S_t)[2]\arw0.
\end{equation*}
Together with the integral $(1,1)$ theorem, this identifies the condition $\alpha_t=0$ with the existence of a class $v\in b+2\Lambda$ of Hodge type $(1,1)$; see \cite[Section 18.1, Equation (1.6), and Section 1.3.3]{huybrechts_k3_lectures}. If $[\sigma_t]$ is the period, this condition is $(v,\sigma_t)=0$ for at least one such $v$.

We first show that these period hyperplanes are dense. Identify the degree-two period domain with the oriented positive two-planes in $h^\perp_{\RR}$, a real quadratic space of signature $(2,19)$. Fix such a plane $\Pi$, and choose a nonzero negative vector $u\in\Pi^\perp\cap h^\perp_{\RR}$. The coset $b+2\Lambda$ is a translate of a full lattice, so there are $v_m\in b+2\Lambda$ with $m^{-1}v_m\arw u$. Put
\begin{equation*}
w_m=v_m-\frac{(v_m,h)}{2}h\in h^\perp_{\RR}.
\end{equation*}
Then $m^{-1}w_m\arw u$, and $w_m$ is nonzero and negative for all sufficiently large $m$. Orthogonal projection of $\Pi$ to $w_m^\perp\cap h^\perp_{\RR}$ gives an oriented positive two-plane $\Pi_m$ with $\Pi_m\arw\Pi$. Thus $\Pi_m$ lies on the hyperplane $(v_m,\sigma)=0$. Since $w_m\ne0$, this is a proper smooth hypersurface of the period domain. On every simply connected neighborhood where the period map has rank nineteen, the inverse images of these hyperplanes are therefore proper analytic hypersurfaces with dense union.

These hypersurfaces are contained in algebraic divisors on all of $T$. To see this without choosing a global integral lift of $\bar b$, consider the variation of Hodge structure $\mathbf{R}^2p_*\ZZ(1)$. The polarization $h$ gives a polarization on this full variation; explicitly, the form
\begin{equation*}
Q(v,w)=(v,h)(w,h)-(v,w)
\end{equation*}
is positive definite on its real Hodge classes. For each positive integer $N$, the space of integral Hodge classes $v$ with $Q(v,v)\leq N$ is algebraic and finite over $T$, by \cite[Theorem 1.1]{cattani_deligne_kaplan_hodge_locus}. The additional condition $v\bmod2=\bar b$ is locally constant on this space, hence selects a union of connected components. Equivalently, reduction modulo two maps to the finite \'etale local system $\mathbf{R}^2p_*\ZZ/2\ZZ$, and we take the inverse image of its section $\bar b$.

Choose an integral class defining one of the proper analytic hypersurfaces above, a point on that hypersurface in the rank-nineteen locus, and $N$ at least its $Q$-norm. An irreducible component of this bounded Hodge-class space through the chosen pair has dimension $\dim T-1$: locally the integral class is fixed and its Hodge locus is precisely that hypersurface. Its image is a closed irreducible algebraic divisor in $T$, since the map is finite. Every point of the image has an integral Hodge class with reduction $\bar b$, so the Kummer sequence proves $\alpha_t=0$ throughout this image. There are only countably many resulting components as $N$ varies. Applying the local density argument on every period chart in the rank-nineteen locus shows that their images have dense union there, and hence in $T$.
\end{proof}

\begin{samepage}
\begin{theorem}[Theorem~\ref{thm:intro_divisorial_rational_fourfolds}]\label{thm:divisorial_rational_fourfolds}
Fix a general hyperplane through the node $x$. In the corresponding parameter space $\mathscr U$ of quadratic equations satisfying Assumption~\ref{ass:generality}, there is a countable collection of closed irreducible divisors $D_i$, whose union is dense in the complex topology, such that every $q\in D_i$ defines a rational one-nodal ordinary GM fourfold $\sfX_q$ and an equivalence
\begin{equation*}
\wt{\Kc u}(\sfX_q)\simeq\dbcoh(\sfS_q),
\end{equation*}
where $\sfS_q$ is its degree-two polarized discriminant $K3$ surface. In particular, these are families of codimension one in the space of ambient quadratic equations singular at $x$.
\end{theorem}
\end{samepage}
\begin{proof}
The universal quadratic form on $\PP^2\times\mathscr U$ has only corank-one singular fibers. Its even Clifford algebra is therefore Azumaya of degree two over the discriminant double cover $\mathscr S$, by the same simple-degeneration argument used in the proof of Theorem~\ref{thm:kuznetsov_clifford_equivalence}. Lemma~\ref{lem:dominant_discriminant_periods} and Lemma~\ref{lem:brauer_vanishing_divisors} give the divisors $D_i$ and the vanishing $\alpha_q=0$ for every $q\in D_i$. The categorical equivalence follows at once from Theorem~\ref{thm:kuznetsov_clifford_equivalence}.

For rationality, put $K=\CC(\PP^2)$ and let $L/K$ be the discriminant extension of the generic quadric surface of $\sfR_q\arw\PP^2$. This is a quadratic field extension, since its branch divisor is a smooth sextic. The even Clifford algebra over $L$ is split because $\alpha_q=0$. Hence one component of the scheme of lines of the generic quadric, viewed over $L$, is a split Severi--Brauer conic and contains an $L$-rational line $\ell$. Its Galois conjugate lies in the other ruling. Their unique intersection is Galois invariant and descends to a $K$-rational point of the quadric. Projection from this point makes the generic quadric rational over $K$. Therefore $\sfR_q$ is rational over $\CC$, and its birational model $\sfX_q$ is rational as well.
\end{proof}

\section{Categorical Torelli theorem for nodal Gushel--Mukai threefolds}\label{sec:categorical_torelli}
Torelli-type problems address the role of Hodge structures and derived categories as invariants of algebraic varieties. For Fano varieties, the categorical Torelli problem usually asks whether the Kuznetsov component determines the isomorphism class, or the birational class, of the variety. Such questions have been studied for cubic threefolds, Gushel--Mukai varieties, and several classes of hypersurfaces; see, for example, \cite{macri:categorical-invarinat-cubic-threefolds,jacovskis2021categorical,pirozhkov2024categorical}. In the singular case it is natural to consider either a smooth admissible subcategory of the Kuznetsov component or a categorical resolution of that component. Here we follow the latter approach.

The geometry of a nodal cubic threefold provides a useful comparison: its blowup at the node is the blowup of $\PP^3$ along a canonical curve of genus four. Orlov's blowup formula \cite{orlovblowup} therefore describes the derived category of the resolution in terms of this curve and exceptional objects. For a nodal Gushel--Mukai threefold, the corresponding geometric model is a conic bundle over $\PP^2$. We use its Clifford category together with the intermediate-Jacobian Torelli theorem for a general Verra threefold. For the relation between derived categories of conic bundles and intermediate Jacobians, see also \cite{bernardara_bolognesi_conic_bundles}.

\begin{theorem}[Theorem~\ref{thm:intro_categorical_torelli}]\label{thm:categorical_torelli}
Let $\sfX$ and $\sfX'$ be general ordinary one-nodal Gushel--Mukai threefolds. If there is a $\CC$-linear exact equivalence
\begin{equation*}
\Phi:\wt{\Kc u}(\sfX)\simeq\wt{\Kc u}(\sfX'),
\end{equation*}
then $\sfX$ and $\sfX'$ are birational.
\end{theorem}

We first explain why an equivalence of the triangulated Clifford categories is sufficient for the intermediate-Jacobian argument. In particular, no compatibility with their $\PP^2$-linear structures is required here.

\begin{lemma}\label{lem:clifford_unique_enhancement}
Let $\mathcal C$ and $\mathcal C'$ be coherent sheaves of algebras on projective varieties over $\CC$. An exact $\CC$-linear equivalence
\begin{equation*}
\dbcoh(\operatorname{coh}\mathcal C)\simeq \dbcoh(\operatorname{coh}\mathcal C')
\end{equation*}
implies that their natural dg enhancements are quasi-equivalent. Consequently, invariants of these enhancements, including their topological $K$-theory with its Hodge structure and Euler pairing when the categories are geometric, agree.
\end{lemma}
\begin{proof}
The category $\operatorname{coh}\mathcal C$ is abelian. By \cite[Theorem A(1)]{canonaco_neeman_stellari_enhancements}, its bounded derived category has a unique $\CC$-linear dg enhancement. Regard the natural enhancement of $\dbcoh(\operatorname{coh}\mathcal C')$ as a second enhancement of $\dbcoh(\operatorname{coh}\mathcal C)$ by composing its identification with the inverse of the given equivalence. Uniqueness gives an isomorphism of the two enhancements in the homotopy category of dg categories. This is all that is needed: the resulting dg equivalence need not be a lift of the particular equivalence with which we started. The last assertion follows from invariance under dg equivalence.
\end{proof}

\begin{lemma}\label{lem:conic_bundle_first_cohomology}
For the conic bundle $\psi:\sfR\arw\PP^2$, one has $\mathbf{R}\psi_*\Oc_{\sfR}\simeq\Oc_{\PP^2}$. In particular,
\begin{equation*}
H^1(\sfR,\Oc_{\sfR})=0,\qquad H^1(\sfR,\ZZ)=0.
\end{equation*}
\end{lemma}
\begin{proof}
Let $\pi:\PP\bigl(\Ec^\vee\bigr)\arw\PP^2$ be the ambient projective bundle. Apply $\mathbf{R}\pi_*$ to
\begin{equation*}
0\arw\Oc(-2h)\arw\Oc\arw\Oc_{\sfR}\arw0.
\end{equation*}
The restriction of $\Oc(-2h)$ to a fiber is $\Oc_{\PP^2}(-2)$, which is acyclic. Hence $\mathbf{R}\pi_*\Oc(-2h)=0$, whereas $\mathbf{R}\pi_*\Oc=\Oc_{\PP^2}$. This proves the first assertion and the coherent-cohomology vanishing. Hodge decomposition gives $H^1(\sfR,\CC)=0$, and $H^1(\sfR,\ZZ)$ is torsion free, so it vanishes as well.
\end{proof}

\begin{lemma}\label{lem:clifford_polarized_jacobian}
Let $f:\sfY\arw\PP^2$ be a flat conic bundle with smooth projective total space, and assume that $H^1(\sfY,\ZZ)=0$ and $H^{3,0}(\sfY)=0$. Then the exact $\CC$-linear equivalence class of its Clifford category $\mathcal B(\sfY/\PP^2)$ determines the principally polarized intermediate Jacobian $\sfJ(\sfY)$.
\end{lemma}
\begin{proof}
The conic-bundle decomposition is
\begin{equation*}
\dbcoh(\sfY)=\langle\mathcal B(\sfY/\PP^2),f^*\dbcoh(\PP^2)\rangle.
\end{equation*}
The odd topological $K$-theory of $\PP^2$ vanishes, as follows either from its cellular decomposition or from its full exceptional collection. Therefore \cite[Proposition 5.23]{perry2020integral} identifies the torsion-free odd topological $K$-theory of $\mathcal B(\sfY/\PP^2)$, with its Hodge structure and Euler pairing, with $H^3(\sfY,\ZZ)$ modulo torsion and its intersection pairing. The hypothesis $H^{3,0}(\sfY)=0$ identifies the resulting polarized complex torus with the principally polarized intermediate Jacobian of $\sfY$. Finally, Lemma~\ref{lem:clifford_unique_enhancement} shows that the dg invariants used here depend only on the exact equivalence class of the triangulated Clifford category.
\end{proof}

\begin{proof}[Proof of Theorem~\ref{thm:categorical_torelli}]
By Theorem~\ref{thm:gm_verra_hyperbolic_comparison} and Corollary~\ref{cor:gm_verra_clifford_equivalence}, there are smooth Verra threefolds $\sfV,\sfV'$ and choices of their conic-bundle projections such that
\begin{equation*}
\begin{alignedat}{2}
\sfX&\sim \sfV, &\qquad \wt{\Kc u}(\sfX)&\simeq\mathcal B(\sfV/\PP^2),\\
\sfX'&\sim \sfV', &\qquad \wt{\Kc u}(\sfX')&\simeq\mathcal B(\sfV'/\PP^2).
\end{alignedat}
\end{equation*}
The dominance assertion in that construction allows us to take $\sfV$ and $\sfV'$ general. This is also the generality appearing in the geometric construction of \cite[Theorem 4.5]{debarre_iliev_manivel_GM_Verra}.

The given equivalence consequently yields an exact equivalence of the two Clifford categories. A smooth Verra threefold is a smooth ample divisor of bidegree $(2,2)$ in $\PP^2\times\PP^2$, so the Lefschetz theorem gives $H^1(\sfV,\ZZ)=0$. Moreover, $-K_{\sfV}=\sfH_1+\sfH_2$, where $\sfH_i$ is the hyperplane class from the $i$-th factor, is ample, and hence $H^{3,0}(\sfV)=0$. The same holds for $\sfV'$. Lemma~\ref{lem:clifford_polarized_jacobian} now gives
\begin{equation*}
\sfJ(\sfV)\simeq \sfJ(\sfV')
\end{equation*}
as principally polarized abelian varieties. The Torelli theorem for a general Verra threefold \cite[Theorem 5.6]{iliev1997theta} implies $\sfV\simeq \sfV'$. Combining this with the displayed birational equivalences proves the theorem.
\end{proof}

\begin{remark}\label{rem:verra_torelli_generality}
The use of Verra threefolds is essential to this proof. The two conic-bundle projections of a general Verra threefold give two discriminant double covers with isomorphic Prym varieties \cite{verra2004prym}; an isomorphism of these Pryms does not in general identify the double covers. The intermediate-Jacobian Torelli theorem reconstructs the Verra threefold itself. The generality hypothesis in Theorem~\ref{thm:categorical_torelli} is retained because this is the scope of the Torelli theorem invoked above.
\end{remark}

\section{Spinor modifications}\label{sec:spinor_modifications}
\subsection{Another birational model for the resolution of nodal Gushel--Mukai threefolds}
In this section we apply the construction of \cite{sasha_spinor_modifications} to the conic bundle $\psi:\sfR\arw\PP^2$. Throughout this subsection the Gushel--Mukai threefold is general, so that the small contraction $\sfR\arw \sfT$ has disjoint exceptional curves with normal bundle $\Oc_{\PP^1}(-1)^{\oplus2}$. We construct an abstract spinor bundle from one of these curves and describe the resulting conic bundle, including its ambient vector bundle and its quadratic form.

We abbreviate a twist by $a\zeta$ on $\PP^2$ to a twist by $a$. Thus $\psi_*\Oc_{\sfR}(h)=\Ec$ and
\begin{equation}\label{eq:spinor_ambient_bundle_resolution}
0\arw\Oc(-1)^{\oplus2}\xrightarrow{M} \Oc^{\oplus4}\oplus\Oc(1)\arw\Ec\arw0.
\end{equation}
Fix an exceptional curve $\sfC\subset \sfR$. The description of the relative Atiyah flop gives
\begin{equation*}
\sfC\simeq\PP^1,\qquad h|_{\sfC}=0,\qquad\zeta|_{\sfC}=c_1\bigl(\Oc_{\sfC}(1)\bigr),\qquad \Nc_{\sfC/\sfR}\simeq\Oc_{\sfC}(-1)^{\oplus2}.
\end{equation*}
In particular, $\psi|_{\sfC}$ identifies $\sfC$ with a line $\sfL\subset\PP^2$. Denote the inclusions by $\eta:\sfC\hookrightarrow \sfR$ and $j:\sfL\hookrightarrow\PP^2$.
\begin{equation*}
\begin{tikzcd}[column sep=large,row sep=large]
\sfC\ar[r,hook,"\eta"]\ar[d,"{\psi|_{\sfC}}"',"\sim"]&\sfR\ar[d,"\psi"]\\
\sfL\ar[r,hook,"j"']&\PP^2.
\end{tikzcd}
\end{equation*}

\begin{lemma}\label{lem:abstract_spinor_bundle}
There is, up to isomorphism and scalar, a unique nonsplit extension
\begin{equation}\label{eq:spinor_serre_extension}
0\arw\Oc_{\sfR}(-h+2\zeta)\arw\Fc \arw\Ic_{\sfC|\sfR}\arw0.
\end{equation}
The sheaf $\Fc$ is a locally free abstract spinor bundle. Moreover,
\begin{equation*}
\mathbf{R}\psi_*\Fc=0,\qquad\det\Fc=\Oc_{\sfR}(-h+2\zeta),\qquad \Fc|_{\sfC}\simeq\Oc_{\sfC}(1)^{\oplus2}.
\end{equation*}
\end{lemma}
\begin{proof}
Put $\Lc_0=\Oc_{\sfR}(-h+2\zeta)$. Since $K_{\sfR}=-h$, relative duality gives
\begin{equation*}
\psi^!\Oc_{\PP^2}(-1)=\Lc_0[1],\qquad \mathbf{R}\psi_*\Lc_0=\Oc_{\PP^2}(-1)[-1].
\end{equation*}
Pushing forward the ideal sequence of $\sfC$ and using $\psi|_{\sfC}:\sfC\simeq \sfL$ gives
\begin{equation*}
\mathbf{R}\psi_*\Ic_{\sfC|\sfR}=\Oc_{\PP^2}(-1).
\end{equation*}
Adjunction therefore yields
\begin{equation*}
\mathbf{R}\operatorname{Hom}_{\sfR}(\Ic_{\sfC|\sfR},\Lc_0) \simeq \mathbf{R}\operatorname{Hom}_{\PP^2}(\Oc(-1),\Oc(-1)[-1]) \simeq\CC[-1].
\end{equation*}
This proves the existence and uniqueness of the nonsplit extension. Its connecting morphism $\Ic_{\sfC|\sfR}\arw \Lc_0[1]$ corresponds, under the same adjunction, to a nonzero scalar endomorphism of $\Oc(-1)$. Applying $\mathbf{R}\psi_*$ to the extension consequently gives an isomorphism between its two outer terms and proves $\mathbf{R}\psi_*\Fc=0$.

We check local freeness. The curve $\sfC$ is a regular codimension-two subscheme and
\begin{equation*}
\Lc_0|_{\sfC}\simeq\Oc_{\sfC}(2)\simeq\det\Nc_{\sfC/\sfR}^{\vee}.
\end{equation*}
Moreover, $H^\bullet(\sfR,\Lc_0)=0$ by the displayed pushforward. The local-to-global Ext sequence identifies
\begin{equation*}
\operatorname{Ext}^1(\Ic_{\sfC|\sfR},\Lc_0) \simeq H^0(\sfC,\Lc_0|_{\sfC}\otimes\det\Nc_{\sfC/\sfR})\simeq\CC.
\end{equation*}
A nonzero class is thus a generator at every point of $\sfC$. Locally, if $(u,v)$ is the regular sequence defining $\sfC$, the extension is the Koszul presentation of $(u,v)$ with its nonzero local Ext generator; its middle term is free of rank two. Away from $\sfC$ local freeness is immediate. This is precisely the Hartshorne--Serre construction.

The determinant follows from Equation~\eqref{eq:spinor_serre_extension} and agrees with $\omega_{\sfR/\PP^2}$ modulo a pullback from the base. Finally, restriction gives an epimorphism
\begin{equation*}
\Fc|_{\sfC}\twoheadrightarrow\Ic_{\sfC|\sfR}/\Ic_{\sfC|\sfR}^2 =\Nc_{\sfC/\sfR}^{\vee}.
\end{equation*}
Both sides have rank two, so it is an isomorphism. This proves all the assertions.
\end{proof}

\begin{lemma}\label{lem:spinor_ambient_bundle_kernel}
The restriction of sections to $\sfC$ induces a surjection $\Ec\arw j_*\Oc_{\sfL}$ with kernel
\begin{equation}\label{eq:spinor_ambient_bundle_kernel}
0\arw \Tc_{\PP^2}(-1)\oplus\Oc(1) \arw\Ec\arw j_*\Oc_{\sfL}\arw0.
\end{equation}
Equivalently, there is an exact sequence
\begin{equation*}
0\arw\Oc(-1)\arw\Oc^{\oplus3}\oplus\Oc(1) \arw\Ec\arw j_*\Oc_{\sfL}\arw0.
\end{equation*}
\end{lemma}
\begin{proof}
We describe the maps, including the generality needed in this computation. Let $W_0$ be the three-dimensional vector space occurring in the Euler presentation of $\Tc_{\PP^2}(-1)$. The linear part of Equation~\eqref{eq:spinor_ambient_bundle_resolution} is induced by the two copies of
\begin{equation*}
\Oc(-1)\arw W_0\otimes\Oc
\end{equation*}
followed by the quotient of $P^\vee\otimes W_0$ by the two-dimensional space $K$ of hyperplane equations. For general hyperplanes, $\PP(K)$ does not meet the Segre variety of decomposable tensors in $\PP\bigl(P^\vee\otimes W_0\bigr)$.

The point to which $\sfC$ is contracted is represented by a decomposable functional on this quotient. Since $\operatorname{Hom}\bigl(\Oc(1),j_*\Oc_{\sfL}\bigr)=0$, the map from the middle term of Equation~\eqref{eq:spinor_ambient_bundle_resolution} is a nonzero constant functional on $\Oc^{\oplus4}$. After choosing bases, its composition with $M$ is $(\ell,0)$, where $\ell=0$ defines $\sfL$. The kernel of the constant functional modulo $\ell$ is $\Oc(-1)\oplus\Oc^{\oplus3}\oplus\Oc(1)$. Cancelling the first source summand against this $\Oc(-1)$ shows that the kernel of $\Ec\arw j_*\Oc_{\sfL}$ is the cokernel of
\begin{equation*}
\Oc(-1)\xrightarrow{(w_1,w_2,w_3,b)}\Oc^{\oplus3}\oplus\Oc(1),
\end{equation*}
where the $w_i$ are linear and $b$ is quadratic. The condition on $K$ implies that the map from the remaining copy of $W_0$ to $(P^\vee\otimes W_0)/ K$ is injective. Hence $w_1,w_2,w_3$ form a basis of $H^0\bigl(\PP^2,\Oc(1)\bigr)$. Write $b=\sum_i b_iw_i$ with $b_i$ linear and subtract these multiples from the last component. The resulting presentation is the Euler sequence together with $\Oc(1)$, proving Equation~\eqref{eq:spinor_ambient_bundle_kernel}.
\end{proof}

\begin{proposition}\label{prop:spinor_endomorphism_bundle}
Set
\begin{equation*}
\mathcal A=\psi_*\mathcal{E}nd^0(\Fc),\qquad \mathcal R_{\Fc}=\psi_*\mathcal{E}nd(\Fc)=\Oc\oplus\mathcal A.
\end{equation*}
These are vector bundles of ranks three and four, respectively, and their higher direct images vanish. There is an exact sequence
\begin{equation}\label{eq:spinor_endomorphism_elementary_transform}
0\arw\mathcal A\arw \Tc_{\PP^2}(-3)\oplus\Oc(-1) \arw j_*\Oc_{\sfL}(-1)^{\oplus2}\arw0.
\end{equation}
Furthermore, there is a point $a\in\PP^2\setminus \sfL$ such that
\begin{equation}\label{eq:spinor_trace_free_bundle}
\mathcal A\simeq\Oc(-2)\oplus\mathcal G_a(-2), \qquad 0\arw\Oc\arw\mathcal G_a\arw\Ic_a\arw0,
\end{equation}
where the latter is the unique nonsplit locally free extension. In coordinates $[x:y:z]$ with $\sfL=(z=0)$ and $a=[0:0:1]$, it has the resolution
\begin{equation*}
0\arw\Oc(-2)\xrightarrow{(x,y,z^2)} \Oc(-1)^{\oplus2}\oplus\Oc\arw\mathcal G_a\arw0.
\end{equation*}
\end{proposition}
\begin{proof}
The local freeness and higher-direct-image vanishing for $\mathcal R_{\Fc}$ follow from the relative tilting property of an abstract spinor bundle \cite[Proposition 2.17]{sasha_spinor_modifications}. We compute its underlying bundle by retaining the maps of the Serre construction.

Twisting Equation~\eqref{eq:spinor_serre_extension} by $\Oc(h-2\zeta)$ and using Equation~\eqref{eq:spinor_ambient_bundle_kernel}, we obtain
\begin{equation*}
\begin{gathered}
0\arw\Oc\arw\psi_*\Fc^\vee \arw \Tc_{\PP^2}(-3)\oplus\Oc(-1)\arw0,\\
\mathbf{R}^i\psi_*\Fc^\vee=0\quad(i>0).
\end{gathered}
\end{equation*}
The extension splits, since
\begin{equation*}
H^1\bigl(\PP^2,\Tc_{\PP^2}^\vee(3)\bigr)=H^1\bigl(\PP^2,\Oc(1)\bigr)=0.
\end{equation*}
Tensoring the Serre extension by $\Fc^\vee$ and composing with the ideal sequence gives
\begin{equation*}
0\arw\Fc\arw\mathcal{E}nd(\Fc) \arw\Fc^\vee\arw\eta_*(\Fc^\vee|_{\sfC})\arw0.
\end{equation*}
Here $\Fc^\vee|_{\sfC}=\Oc_{\sfC}(-1)^{\oplus2}$. Since $\mathbf{R}\psi_*\Fc=0$, pushing forward yields
\begin{equation*}
0\arw\mathcal R_{\Fc}\arw\psi_*\Fc^\vee \arw j_*\Oc_{\sfL}(-1)^{\oplus2}\arw0.
\end{equation*}
The identity endomorphism maps to the section generating the displayed $\Oc$ in $\psi_*\Fc^\vee$: this follows directly by interpreting the map $\mathcal{E}nd(\Fc)\arw\Fc^\vee$ as composition with $\Fc\arw\Ic_{\sfC}\subset\Oc$. Quotienting by that summand gives Equation~\eqref{eq:spinor_endomorphism_elementary_transform}.

We next identify its kernel. The restriction of the middle term to $\sfL$ is
\begin{equation*}
\bigl(\Tc_{\PP^2}(-3)\oplus\Oc(-1)\bigr)|_{\sfL} \simeq\Oc_{\sfL}(-1)\oplus\Oc_{\sfL}(-2)\oplus\Oc_{\sfL}(-1).
\end{equation*}
Surjectivity onto $\Oc_{\sfL}(-1)^{\oplus2}$ forces the constant matrix from the two $\Oc_{\sfL}(-1)$ summands to have rank two. Indeed, if its rank were one, the remaining linear entry from $\Oc_{\sfL}(-2)$ would vanish somewhere on $\sfL$; if its rank were zero, surjectivity would also be impossible.

Choose a basis of the target such that the last $\Oc(-1)$ maps to the second summand by restriction. The restriction map
\begin{equation*}
\operatorname{Hom}\bigl(\Tc_{\PP^2}(-3),\Oc(-1)\bigr) \arw\operatorname{Hom}\bigl(\Tc_{\PP^2}(-3),j_*\Oc_{\sfL}(-1)\bigr)
\end{equation*}
is an isomorphism: its kernel and cokernel are $H^0\bigl(\Tc_{\PP^2}^\vee(1)\bigr)$ and $H^1\bigl(\Tc_{\PP^2}^\vee(1)\bigr)$, both zero by the Euler sequence. An automorphism of the middle term therefore eliminates its second component on $\Tc_{\PP^2}(-3)$. We are left with the direct sum of $\Oc(-1)\arw j_*\Oc_{\sfL}(-1)$ and a surjection
\begin{equation*}
r:\Tc_{\PP^2}(-3)\arw j_*\Oc_{\sfL}(-1).
\end{equation*}
In the Euler presentation of $\Tc_{\PP^2}(-3)$, such a map is represented by a row of linear forms annihilating $(x,y,z)$, hence by a skew-symmetric $3\times3$ matrix. Its unique zero in $\PP^2$ is a point $a$; surjectivity along $\sfL$ is equivalent to $a\notin \sfL$. Choose coordinates as in the statement. The row is then $(y,-x,0)$ modulo $z$, up to a nonzero scalar.

For completeness, its kernel can be read off without a splitting argument. Before quotienting by the Euler vector $(x,y,z)$, the kernel of the row modulo $z$ is generated by
\begin{equation*}
(x,y,0),\quad(z,0,0),\quad(0,z,0),\quad(0,0,1),
\end{equation*}
with relation $z(x,y,0)-x(z,0,0)-y(0,z,0)=0$. The Euler vector identifies the first generator with $-z(0,0,1)$. Thus
\begin{equation*}
0\arw\Oc(-4)\xrightarrow{(x,y,z^2)} \Oc(-3)^{\oplus2}\oplus\Oc(-2)\arw\ker r\arw0.
\end{equation*}
The entries have no common projective zero, so this cokernel is locally free. After twisting by $2$, the image of the last $\Oc$ has quotient $\Ic_a$, giving the nonsplit extension in Equation~\eqref{eq:spinor_trace_free_bundle}. The other direct summand of $\mathcal A$ is the kernel $\Oc(-2)$ of $\Oc(-1)\arw j_*\Oc_{\sfL}(-1)$. This proves the proposition.
\end{proof}

\begin{remark}
The preceding computation gives
\begin{equation*}
\operatorname{ch}(\mathcal A)=3-6\zeta+5\zeta^2,\qquad \operatorname{ch}(\mathcal R_{\Fc})=4-6\zeta+5\zeta^2.
\end{equation*}
In particular, $\mathcal A$ is not a sum of three line bundles. Such a splitting with degrees $a_1,a_2,a_3$ would require $\sum a_i=-6$ and $\sum a_i^2=10$, whereas $\sum a_i^2\geq12$. The same exact sequence gives $\operatorname{Ext}^i(\Fc,\Fc)\simeq\CC$ for $i=0,1$ and zero otherwise.
\end{remark}

\begin{proposition}\label{prop:spinor_modified_conic_bundle}
There is a flat conic bundle $\sfU\arw\PP^2$ with smooth total space, birational to $\sfR$ over $\PP^2$, with ambient projective bundle
\begin{equation*}
\sfU\subset\PP_{\PP^2}\bigl(\Oc\oplus\mathcal G_a^\vee\bigr).
\end{equation*}
Its equation is a specific section of $\Oc(2\theta+2\zeta)$, where $\theta$ is the relative hyperplane class. The quadratic form is determined by the multiplication on $\mathcal R_{\Fc}=\psi_*\mathcal{E}nd(\Fc)$, as described below. Its discriminant is the same sextic curve as the discriminant of $\sfR$.
\end{proposition}
\begin{proof}
The commutator of endomorphisms defines
\begin{equation*}
\det\mathcal A\otimes\mathcal A^\vee\simeq\wedge^2\mathcal A \xrightarrow{[-,-]}\mathcal A.
\end{equation*}
Since $\mathcal R_{\Fc}$ is a pointwise even Clifford algebra, the corresponding tensor is symmetric. It gives the nowhere-zero quadratic form
\begin{equation}\label{eq:spinor_modified_quadratic_form}
q_{\Fc}:\det\mathcal A\arw\operatorname{Sym}^2\mathcal A
\end{equation}
and recovers the multiplication on $\mathcal R_{\Fc}$ by \cite[Proposition 2.7]{sasha_spinor_modifications}. This describes the equation explicitly from the Serre bundle: if $a_1,a_2,a_3$ is a local frame of $\mathcal A$, take as the three columns of a matrix the coefficients of
\begin{equation*}
[a_2,a_3],\qquad[a_3,a_1],\qquad[a_1,a_2]
\end{equation*}
in that frame. The matrix is symmetric, and its associated quadratic polynomial is the local equation of the modification. These equations glue by Equation~\eqref{eq:spinor_modified_quadratic_form}.

By Equation~\eqref{eq:spinor_trace_free_bundle}, $\det\mathcal A=\Oc(-6)$ and $\mathcal A^\vee=(\Oc\oplus\mathcal G_a^\vee)(2)$. Twisting the quadratic data by $\Oc(-2)$ identifies the same conic bundle with the zero locus of
\begin{equation*}
\Oc(-2)\arw\operatorname{Sym}^2(\Oc\oplus\mathcal G_a)
\end{equation*}
in $\PP\bigl(\Oc\oplus\mathcal G_a^\vee\bigr)$; its divisor class is $2\theta+2\zeta$. Kuznetsov's spinor-modification theorem \cite[Theorem 1.2]{sasha_spinor_modifications} proves flatness and identifies its even Clifford algebra with $\mathcal R_{\Fc}$ in the normalized presentation. Smoothness follows from \cite[Proposition 2.22]{sasha_spinor_modifications}. The two generic even Clifford algebras are Morita equivalent quaternion algebras over $\CC(\PP^2)$, hence are isomorphic; their Severi--Brauer conics are isomorphic. Thus $\sfU$ and $\sfR$ are birational over the base.

Finally, the fiberwise description of abstract spinor bundles identifies $\mathcal R_{\Fc}\otimes\CC(s)$ with the even Clifford algebra of the original conic over every geometric point $s$. It follows that the two conics are singular at exactly the same points of the base. The discriminant of Equation~\eqref{eq:spinor_modified_quadratic_form} has degree six, because its line-bundle class is
\begin{equation*}
(\det\mathcal A^\vee)^{-2}\otimes(\det\mathcal A)^{-3} =\Oc(6).
\end{equation*}
The original discriminant is a smooth sextic; equality of supports and degrees therefore proves equality as divisors.
\end{proof}

\begin{corollary}\label{cor:spinor_clifford_equivalence}
There is a $\PP^2$-linear Fourier--Mukai equivalence
\begin{equation*}
\mathcal B(\sfR/\PP^2)\simeq\mathcal B(\sfU/\PP^2).
\end{equation*}
In particular, $\sfU$ is a birational model of $\sfX$ with Clifford category equivalent to $\wt{\Kc u}(\sfX)$.
\end{corollary}
\begin{proof}
This is the categorical assertion of \cite[Theorem 1.2]{sasha_spinor_modifications}, followed by the equivalence of Theorem~\ref{thm:kuznetsov_clifford_equivalence}.
\end{proof}

\subsection{Spinor modification of Verra threefolds}
Let $\sfV\subset\PP^2\times\PP^2$ be a smooth Verra threefold, and let $\pi_i:\PP^2\times\PP^2\arw\PP^2$ be the $i$-th projection, with restriction $f_i:\sfV\arw\PP^2$. These maps fit into the diagram
\begin{equation*}
\begin{tikzcd}[column sep=large,row sep=large]
&\sfV\ar[d,hook]\ar[dl,"f_1"']\ar[dr,"f_2"]&\\
\PP^2&\PP^2\times\PP^2\ar[l,"\pi_1"]\ar[r,"\pi_2"']&\PP^2.
\end{tikzcd}
\end{equation*}
We keep track of the chosen projection when writing its Clifford category. The spinor-modification theorem gives the following form of birational Torelli without a generality hypothesis on the conic bundles.

\begin{proposition}\label{prop:base_linear_birational_torelli}
Let $f:\sfY\arw\PP^2$ and $f':\sfY'\arw\PP^2$ be flat conic bundles with smooth total spaces and smooth generic fibers. Suppose there is a $\PP^2$-linear, $t$-exact Fourier--Mukai equivalence
\begin{equation*}
\operatorname{Ker}(\mathbf{R}f_*)\simeq\operatorname{Ker}(\mathbf{R}f'_*),
\end{equation*}
where the kernel categories have the induced $t$-structures. Then $\sfY$ and $\sfY'$ are birational over $\PP^2$. In particular, this applies to the chosen conic-bundle structures on two smooth Verra threefolds.
\end{proposition}
\begin{proof}
The converse assertion of \cite[Theorem 1.2]{sasha_spinor_modifications} realizes $\sfY'$ as a spinor modification of $\sfY$. The generic Clifford algebras are therefore Morita equivalent quaternion algebras over $\CC(\PP^2)$, and consequently are isomorphic. Their conics are isomorphic over that field, proving birationality over the base, as in \cite[Lemma 2.21]{sasha_spinor_modifications}.
\end{proof}

The hypotheses concerning base-linearity and $t$-exactness are part of this statement. Theorem~\ref{thm:categorical_torelli}, which assumes only an abstract exact equivalence, was proved instead by the intermediate-Jacobian argument.

We finish by constructing an abstract spinor bundle on a Verra threefold directly from the diagonal. In this case the modification is trivial.

\begin{lemma}\label{lem:verra_canonical_spinor_bundle}
Let $\sfD\subset\PP^2\times\PP^2$ be the diagonal and assume that $\sfC=\sfV\cap\sfD$ has the expected codimension two in $\sfV$. Then
\begin{equation}\label{eq:verra_diagonal_spinor_sequence}
\begin{gathered}
0\arw\Oc_{\sfV}(-2\sfH_1-\sfH_2)\arw\Fc_{\sfD} \arw\Ic_{\sfC|\sfV}\arw0,\\
\Fc_{\sfD}=\bigl(\pi_1^*\Oc(-1)\otimes\pi_2^*\Tc_{\PP^2}^\vee(1)\bigr)|_{\sfV}.
\end{gathered}
\end{equation}
The bundle $\Fc_{\sfD}$ is a base twist of a canonical abstract spinor bundle for $f_1$, and its spinor modification is isomorphic to $\sfV$ over $\PP^2$.
\end{lemma}
\begin{proof}
The diagonal is the zero locus of the tautological section of the rank-two bundle
\begin{equation*}
\pi_1^*\Oc(1)\otimes\pi_2^*\Tc_{\PP^2}(-1).
\end{equation*}
Indeed, the section is induced by the map from the tautological line on the first factor to the universal quotient on the second, and vanishes exactly when the two lines coincide. The expected-codimension hypothesis makes its restriction to $\sfV$ regular. Its Koszul resolution gives Equation~\eqref{eq:verra_diagonal_spinor_sequence}, including the stated determinant.

Restricting the Euler sequence on the second factor gives
\begin{equation*}
0\arw\pi_2^*\Tc_{\PP^2}^\vee(1)|_{\sfV} \arw\Oc_{\sfV}^{\oplus3}\arw\Oc_{\sfV}(\sfH_2)\arw0.
\end{equation*}
For $f_1$ the pushforward of $\Oc_{\sfV}(\sfH_2)$ is $\Oc_{\PP^2}^{\oplus3}$, and its higher direct images vanish. Thus the left-hand bundle lies in $\operatorname{Ker}(\mathbf{R}f_{1*})$ and is the canonical spinor bundle $\Fc^{-1}$ of \cite[Lemma 2.16]{sasha_spinor_modifications}, up to the normalization of the quadratic data. Tensoring by $f_1^*\Oc(-1)$ gives $\Fc_{\sfD}$. Canonical spinor bundles recover the original even Clifford algebra and its conic bundle; tensoring by a line bundle from the base does not change the endomorphism algebra. Therefore the modification is $\sfV$ itself.
\end{proof}

\begin{remark}
The diagonal construction suggests asking whether every abstract spinor bundle on a general Verra threefold whose modification is again a Verra threefold is canonical, modulo tensoring by a line bundle from the base. This is a question about spinor modifications with a fixed base. A positive answer would give an isomorphism conclusion for the equivalences satisfying the hypotheses of Proposition~\ref{prop:base_linear_birational_torelli}; it would not by itself identify those hypotheses for an arbitrary equivalence of triangulated Clifford categories.
\end{remark}

\bibliographystyle{alpha}
\bibliography{biblio}

@article{BayerLahozMacriNuerPerryStellari2021,
  author = {Arend Bayer and Mart\'{\i} Lahoz and Emanuele Macr\`{\i} and Howard Nuer and Alexander Perry and Paolo Stellari},
  title = {Stability conditions in families},
  journal = {Publications Math\'{e}matiques de l'IH\'{E}S},
  volume = {133},
  pages = {157--325},
  year = {2021},
  doi = {10.1007/s10240-021-00124-6},
  eprint = {1902.08184},
  archivePrefix = {arXiv},
  primaryClass = {math.AG},
  url = {https://arxiv.org/abs/1902.08184},
}

@article{DebarreKuznetsov2019GushelMukai,
  author = {Debarre, Olivier and Kuznetsov, Alexander},
  title = {{Gushel}--{M}ukai varieties: linear spaces and periods},
  journal = {Kyoto Journal of Mathematics},
  volume = {59},
  number = {4},
  pages = {897--953},
  year = {2019},
  doi = {10.1215/21562261-2019-0030},
  url = {https://doi.org/10.1215/21562261-2019-0030},
}

@article{IlievManivel2011FanoDegreeTen,
  author = {Iliev, Atanas and Manivel, Laurent},
  title = {{Fano} manifolds of degree ten and {EPW} sextics},
  journal = {Annales Scientifiques de l'{\'E}cole Normale Sup{\'e}rieure},
  series = {4},
  volume = {44},
  number = {3},
  pages = {393--426},
  year = {2011},
  doi = {10.24033/asens.2146},
  url = {https://doi.org/10.24033/asens.2146},
}

@article{LahozLehnMacriStellari2018,
  author = {Mart\'{\i} Lahoz and Manfred Lehn and Emanuele Macr\`{\i} and Paolo Stellari},
  title = {Generalized twisted cubics on a cubic fourfold as a moduli space of stable objects},
  journal = {Journal de Math\'{e}matiques Pures et Appliqu\'{e}es},
  series = {9},
  volume = {114},
  pages = {85--117},
  year = {2018},
  doi = {10.1016/j.matpur.2017.09.004},
  eprint = {1609.04573},
  archivePrefix = {arXiv},
  primaryClass = {math.AG},
  url = {https://arxiv.org/abs/1609.04573},
}

@article{Logachev2012FanoGenus6,
  author = {Logachev, Dmitry},
  title = {{Fano} threefolds of genus 6},
  journal = {Asian Journal of Mathematics},
  volume = {16},
  number = {3},
  pages = {515--559},
  year = {2012},
  doi = {10.4310/AJM.2012.v16.n3.a9},
  eprint = {math/0407147},
  archivePrefix = {arXiv},
  primaryClass = {math.AG},
  url = {https://arxiv.org/abs/math/0407147},
}

@article{Nagel1998GeneralizedHodge,
  author = {Nagel, J.},
  title = {The generalized {H}odge conjecture for the quadratic complex of lines in projective four-space},
  journal = {Mathematische Annalen},
  volume = {312},
  number = {2},
  pages = {387--401},
  year = {1998},
  doi = {10.1007/s002080050227},
  url = {https://doi.org/10.1007/s002080050227},
}

@article{addington_donovan_meachan_flops,
  author = {Addington, Nicolas and Donovan, Will and Meachan, Ciaran},
  title = {{Mukai} flops and {$\mathbb P$}-twists},
  journal = {Journal f\"ur die reine und angewandte Mathematik},
  volume = {748},
  pages = {227--240},
  year = {2019},
  doi = {10.1515/crelle-2016-0024},
  eprint = {1507.02595},
  archivePrefix = {arXiv},
  primaryClass = {math.AG},
  url = {https://arxiv.org/abs/1507.02595},
  note = {Numbered references use the expanded arXiv version},
}

@article{auel_bernardara_bolognesi_quadric_fibrations,
  author = {Asher Auel and Marcello Bernardara and Michele Bolognesi},
  title = {Fibrations in complete intersections of quadrics, {Clifford} algebras, derived categories, and rationality problems},
  journal = {Journal de Math{\'e}matiques Pures et Appliqu{\'e}es},
  series = {9},
  volume = {102},
  number = {1},
  pages = {249--291},
  year = {2014},
  doi = {10.1016/j.matpur.2013.11.009},
  eprint = {1109.6938},
  archivePrefix = {arXiv},
  primaryClass = {math.AG},
  url = {https://arxiv.org/abs/1109.6938},
}

@article{bini_kapustka2,
  author = {Gilberto Bini and Grzegorz Kapustka and Micha{\l} Kapustka},
  title = {Symmetric locally free resolutions and rationality problems},
  journal = {Communications in Contemporary Mathematics},
  volume = {25},
  number = {8},
  pages = {2250033},
  year = {2023},
  doi = {10.1142/S021919972250033X},
  eprint = {2005.02092},
  archivePrefix = {arXiv},
  primaryClass = {math.AG},
  url = {https://arxiv.org/abs/2005.02092},
}

@misc{bondal_orlov_flop,
  author = {Bondal, Alexey and Orlov, Dmitri},
  title = {Semiorthogonal decomposition for algebraic varieties},
  year = {1995},
  eprint = {alg-geom/9506012},
  archivePrefix = {arXiv},
  primaryClass = {alg-geom},
  url = {https://arxiv.org/abs/alg-geom/9506012},
}

@article{canonaco_neeman_stellari_enhancements,
  author = {Canonaco, Alberto and Neeman, Amnon and Stellari, Paolo},
  title = {Uniqueness of enhancements for derived and geometric categories},
  journal = {Forum of Mathematics, Sigma},
  volume = {10},
  pages = {e92},
  year = {2022},
  doi = {10.1017/fms.2022.82},
  eprint = {2101.04404},
  archivePrefix = {arXiv},
  primaryClass = {math.AG},
  url = {https://arxiv.org/abs/2101.04404},
}

@article{cattani_et_al,
  author = {Warren Cattani and Franco Giovenzana and Shengxuan Liu and Pablo Magni and Luigi Martinelli and Laura Pertusi and Jieao Song},
  title = {Kernels of categorical resolutions of nodal singularities},
  journal = {Rendiconti del Circolo Matematico di Palermo},
  series = {2},
  volume = {72},
  number = {6},
  pages = {3077--3105},
  year = {2023},
  doi = {10.1007/s12215-023-00895-3},
  url = {https://doi.org/10.1007/s12215-023-00895-3},
}

@article{debarre_iliev_manivel_GM_Verra,
  author = {Olivier Debarre and Atanas Iliev and Laurent Manivel},
  title = {On nodal prime {Fano} threefolds of degree 10},
  journal = {Science China Mathematics},
  volume = {54},
  number = {8},
  pages = {1591--1609},
  year = {2011},
  doi = {10.1007/s11425-011-4182-0},
  url = {https://doi.org/10.1007/s11425-011-4182-0},
}

@article{guo2024conics,
  author = {Guo, Hanfei and Liu, Zhiyu and Zhang, Shizhuo},
  title = {Conics on {Gushel}--{Mukai} fourfolds, {EPW} sextics and {Bridgeland} moduli spaces},
  journal = {Mathematical Research Letters},
  volume = {31},
  number = {4},
  pages = {1061--1106},
  year = {2024},
  doi = {10.4310/MRL.241119003311},
  url = {https://doi.org/10.4310/MRL.241119003311},
}

@book{huyb-book-FM,
  author = {Huybrechts, Daniel},
  title = {{Fourier}--{M}ukai transforms in algebraic geometry},
  series = {Oxford Mathematical Monographs},
  pages = {viii+307},
  publisher = {Oxford University Press},
  address = {Oxford},
  year = {2006},
  isbn = {978-0-19-929686-6},
  doi = {10.1093/acprof:oso/9780199296866.001.0001},
  url = {https://doi.org/10.1093/acprof:oso/9780199296866.001.0001},
}

@article{iliev1997theta,
  author = {Atanas Iliev},
  title = {The theta divisor of bidegree {$(2,2)$} threefold in {$\mathbb{P}^{2}\times\mathbb{P}^{2}$}},
  journal = {Pacific Journal of Mathematics},
  volume = {180},
  number = {1},
  pages = {57--88},
  year = {1997},
  url = {https://msp.org/pjm/1997/180-1/pjm-v180-n1-p05-p.pdf},
}

@article{iliev_kapustka_kapustka_ranestad_verra,
  author = {Iliev, Atanas and Kapustka, Grzegorz and Kapustka, Micha{\l} and Ranestad, Kristian},
  title = {Hyper-{K}\"ahler fourfolds and {Kummer} surfaces},
  journal = {Proceedings of the London Mathematical Society},
  series = {3},
  volume = {115},
  number = {6},
  pages = {1276--1316},
  year = {2017},
  doi = {10.1112/plms.12063},
  eprint = {1603.00403},
  archivePrefix = {arXiv},
  primaryClass = {math.AG},
  url = {https://doi.org/10.1112/plms.12063},
}

@article{jacovskis2021categorical,
  author = {Jacovskis, Augustinas and Lin, Xun and Liu, Zhiyu and Zhang, Shizhuo},
  title = {Categorical {T}orelli theorems for {G}ushel--{M}ukai threefolds},
  journal = {Journal of the London Mathematical Society},
  series = {2},
  volume = {109},
  number = {3},
  pages = {e12878},
  year = {2024},
  doi = {10.1112/jlms.12878},
  url = {https://doi.org/10.1112/jlms.12878},
}

@incollection{kuznetsov_cubic,
  author = {Alexander Kuznetsov},
  title = {Derived categories of cubic fourfolds},
  editor = {Bogomolov, Fedor and Tschinkel, Yuri},
  booktitle = {Cohomological and Geometric Approaches to Rationality Problems},
  series = {Progress in Mathematics},
  volume = {282},
  pages = {219--243},
  publisher = {Birkh\"auser},
  address = {Boston, MA},
  year = {2010},
  doi = {10.1007/978-0-8176-4934-0_9},
  eprint = {0808.3351},
  archivePrefix = {arXiv},
  primaryClass = {math.AG},
  url = {https://arxiv.org/abs/0808.3351},
}

@article{kuznetsov_perry_categorical_cones,
  author = {Kuznetsov, Alexander and Perry, Alexander},
  title = {Categorical cones and quadratic homological projective duality},
  journal = {Annales Scientifiques de l'\'{E}cole Normale Sup\'{e}rieure},
  series = {4},
  volume = {56},
  number = {1},
  pages = {1--57},
  year = {2023},
  doi = {10.24033/asens.2527},
  url = {https://doi.org/10.24033/asens.2527},
}

@incollection{kuznetsov_prokhorov,
  author = {Kuznetsov, Alexander and Prokhorov, Yuri},
  title = {Rationality of {Mukai} Varieties over Non-closed Fields},
  editor = {Farkas, Gavril and van der Geer, Gerard and Shen, Mingmin and Taelman, Lenny},
  booktitle = {Rationality of Varieties},
  series = {Progress in Mathematics},
  volume = {342},
  pages = {249--290},
  publisher = {Springer International Publishing},
  address = {Cham},
  year = {2021},
  isbn = {978-3-030-75421-1},
  doi = {10.1007/978-3-030-75421-1_10},
  url = {https://doi.org/10.1007/978-3-030-75421-1_10},
}

@article{kuznetsov_quadric_fibrations_intersection_quadrics,
  author = {Alexander Kuznetsov},
  title = {Derived categories of quadric fibrations and intersections of quadrics},
  journal = {Advances in Mathematics},
  volume = {218},
  number = {5},
  pages = {1340--1369},
  year = {2008},
  doi = {10.1016/j.aim.2008.03.007},
  url = {https://doi.org/10.1016/j.aim.2008.03.007},
}

@article{kuznetsov_resolutions_of_singularities,
  author = {Kuznetsov, Alexander},
  title = {{Lefschetz} decompositions and categorical resolutions of singularities},
  journal = {Selecta Mathematica. New Series},
  volume = {13},
  number = {4},
  pages = {661--696},
  year = {2008},
  doi = {10.1007/s00029-008-0052-1},
  eprint = {math/0609240},
  archivePrefix = {arXiv},
  primaryClass = {math.AG},
  url = {https://arxiv.org/abs/math/0609240},
}

@article{kuznetsovperry,
  author = {Alexander Kuznetsov and Alexander Perry},
  title = {Derived categories of {Gushel}--{Mukai} varieties},
  journal = {Compositio Mathematica},
  volume = {154},
  number = {7},
  pages = {1362--1406},
  year = {2018},
  doi = {10.1112/S0010437X18007091},
  eprint = {1605.06568},
  archivePrefix = {arXiv},
  primaryClass = {math.AG},
  url = {https://arxiv.org/abs/1605.06568},
}

@article{macri:categorical-invarinat-cubic-threefolds,
  author = {Bernardara, Marcello and Macr\`{\i}, Emanuele and Mehrotra, Sukhendu and Stellari, Paolo},
  title = {A categorical invariant for cubic threefolds},
  journal = {Advances in Mathematics},
  volume = {229},
  number = {2},
  pages = {770--803},
  year = {2012},
  doi = {10.1016/j.aim.2011.10.007},
  url = {https://doi.org/10.1016/j.aim.2011.10.007},
}

@article{mella1999existence,
  author = {Mella, Massimiliano},
  title = {Existence of Good Divisors on {Mukai} Varieties},
  journal = {Journal of Algebraic Geometry},
  volume = {8},
  number = {2},
  pages = {197--206},
  year = {1999},
  eprint = {alg-geom/9705012},
  archivePrefix = {arXiv},
  primaryClass = {alg-geom},
  url = {https://arxiv.org/abs/alg-geom/9705012},
}

@article{mukai1989biregular,
  author = {Mukai, Shigeru},
  title = {Biregular Classification of {Fano} 3-Folds and {Fano} Manifolds of Coindex {3}},
  journal = {Proceedings of the National Academy of Sciences of the United States of America},
  volume = {86},
  number = {9},
  pages = {3000--3002},
  year = {1989},
  doi = {10.1073/pnas.86.9.3000},
  url = {https://doi.org/10.1073/pnas.86.9.3000},
}

@incollection{mukai1993curves,
  author = {Mukai, Shigeru},
  title = {Curves and {Grassmannians}},
  booktitle = {Algebraic Geometry and Related Topics (Inchon, 1992)},
  series = {Conference Proceedings \& Lecture Notes in Algebraic Geometry},
  volume = {I},
  pages = {19--40},
  publisher = {International Press},
  address = {Cambridge, MA},
  year = {1993},
}

@article{mukai1995new,
  author = {Mukai, Shigeru},
  title = {New Development of Theory of {Fano} 3-folds: Vector Bundle Method and Moduli Problem},
  journal = {Sugaku},
  volume = {47},
  number = {2},
  pages = {125--144},
  year = {1995},
  note = {In Japanese; English translation: Sugaku Expositions 15 (2002), no. 2, 125--150},
}

@article{orlovblowup,
  author = {Orlov, Dmitri O.},
  title = {Projective bundles, monoidal transformations, and derived categories of coherent sheaves},
  journal = {Russian Academy of Sciences. Izvestiya Mathematics},
  volume = {41},
  number = {1},
  pages = {133--141},
  year = {1993},
  doi = {10.1070/IM1993v041n01ABEH002182},
  url = {https://doi.org/10.1070/IM1993v041n01ABEH002182},
}

@article{perry2020integral,
  author = {Perry, Alexander},
  title = {The integral {H}odge conjecture for two-dimensional {C}alabi--{Y}au categories},
  journal = {Compositio Mathematica},
  volume = {158},
  number = {2},
  pages = {287--333},
  year = {2022},
  doi = {10.1112/s0010437x22007266},
  url = {https://doi.org/10.1112/s0010437x22007266},
}

@article{perry_pertusi_zhao,
  author = {Perry, Alexander and Pertusi, Laura and Zhao, Xiaolei},
  title = {Stability conditions and moduli spaces for {Kuznetsov} components of {Gushel}--{Mukai} varieties},
  journal = {Geometry \& Topology},
  volume = {26},
  number = {7},
  pages = {3055--3121},
  year = {2022},
  doi = {10.2140/gt.2022.26.3055},
  url = {https://doi.org/10.2140/gt.2022.26.3055},
}

@article{pirozhkov2024categorical,
  author = {Pirozhkov, Dmitrii},
  title = {A categorical {T}orelli theorem for hypersurfaces},
  journal = {Bulletin of the London Mathematical Society},
  volume = {56},
  number = {10},
  pages = {3075--3089},
  year = {2024},
  doi = {10.1112/blms.13117},
  url = {https://doi.org/10.1112/blms.13117},
}

@article{sasha_spinor_modifications,
  author = {Kuznetsov, Alexander},
  title = {Spinor modifications of conic bundles and derived categories of 1-nodal {F}ano threefolds},
  journal = {Proceedings of the Steklov Institute of Mathematics},
  volume = {329},
  number = {1},
  pages = {88--116},
  year = {2025},
  doi = {10.1134/S0081543825600759},
  eprint = {2502.02082v2},
  archivePrefix = {arXiv},
  primaryClass = {math.AG},
  url = {https://arxiv.org/abs/2502.02082v2},
  note = {Numbered references use arXiv:2502.02082v2},
}

@inproceedings{verra2004prym,
  author = {Alessandro Verra},
  title = {The {Prym} map has degree two on plane sextics},
  booktitle = {The {Fano} Conference},
  pages = {735--759},
  publisher = {Universit{\`a} di Torino},
  address = {Turin, Italy},
  year = {2004},
}

@article{xie_quadric_bundles,
  author = {Xie, Fei},
  title = {Residual categories of quadric surface bundles},
  journal = {Journal f\"ur die reine und angewandte Mathematik},
  volume = {796},
  pages = {161--199},
  year = {2023},
  doi = {10.1515/crelle-2022-0092},
  url = {https://doi.org/10.1515/crelle-2022-0092},
}

@article{Ran2007NormalBundles,
  author = {Ran, Ziv},
  title = {Normal bundles of rational curves in projective spaces},
  journal = {Asian Journal of Mathematics},
  volume = {11},
  number = {4},
  pages = {567--608},
  year = {2007},
  doi = {10.4310/AJM.2007.v11.n4.a3},
  eprint = {math/0310054},
  archivePrefix = {arXiv},
  primaryClass = {math.AG},
  url = {https://arxiv.org/abs/math/0310054},
}

@article{kapranov_homogeneous_spaces,
  author = {Kapranov, Mikhail M.},
  title = {On the derived categories of coherent sheaves on some homogeneous spaces},
  journal = {Inventiones Mathematicae},
  volume = {92},
  number = {3},
  pages = {479--508},
  year = {1988},
  doi = {10.1007/BF01393744},
  url = {https://doi.org/10.1007/BF01393744},
}

@misc{stacks_project,
  author = {{The Stacks Project Authors}},
  title = {The {Stacks Project}},
  year = {2026},
  url = {https://stacks.math.columbia.edu},
}

@article{kuznetsov_hyperbolic_equivalence,
  author = {Kuznetsov, Alexander},
  title = {Quadric bundles and hyperbolic equivalence},
  journal = {Geometry \& Topology},
  volume = {28},
  number = {3},
  pages = {1287--1339},
  year = {2024},
  doi = {10.2140/gt.2024.28.1287},
  eprint = {2108.01546},
  archivePrefix = {arXiv},
  primaryClass = {math.AG},
  url = {https://arxiv.org/abs/2108.01546},
}

@article{cattani_deligne_kaplan_hodge_locus,
  author = {Cattani, Eduardo and Deligne, Pierre and Kaplan, Aroldo},
  title = {On the locus of {Hodge} classes},
  journal = {Journal of the American Mathematical Society},
  volume = {8},
  number = {2},
  pages = {483--506},
  year = {1995},
  doi = {10.1090/S0894-0347-1995-1273413-2},
  eprint = {alg-geom/9402009},
  archivePrefix = {arXiv},
  primaryClass = {alg-geom},
  url = {https://doi.org/10.1090/S0894-0347-1995-1273413-2},
}

@book{huybrechts_k3_lectures,
  author = {Huybrechts, Daniel},
  title = {Lectures on {K3} Surfaces},
  series = {Cambridge Studies in Advanced Mathematics},
  volume = {158},
  publisher = {Cambridge University Press},
  address = {Cambridge},
  year = {2016},
  isbn = {978-1-107-15304-2},
  doi = {10.1017/CBO9781316594193},
  url = {https://doi.org/10.1017/CBO9781316594193},
}

@article{beauville_determinantal_hypersurfaces,
  author = {Beauville, Arnaud},
  title = {Determinantal hypersurfaces},
  journal = {Michigan Mathematical Journal},
  volume = {48},
  number = {1},
  pages = {39--64},
  year = {2000},
  doi = {10.1307/mmj/1030132707},
  eprint = {math/9910030},
  archivePrefix = {arXiv},
  primaryClass = {math.AG},
  url = {https://doi.org/10.1307/mmj/1030132707},
}

@article{addington_thomas_cubic_fourfolds,
  author = {Addington, Nicolas and Thomas, Richard P.},
  title = {{Hodge} theory and derived categories of cubic fourfolds},
  journal = {Duke Mathematical Journal},
  volume = {163},
  number = {10},
  pages = {1885--1927},
  year = {2014},
  doi = {10.1215/00127094-2738639},
  eprint = {1211.3758},
  archivePrefix = {arXiv},
  primaryClass = {math.AG},
  url = {https://arxiv.org/abs/1211.3758},
}

@article{gushel_fano_genus6,
  author = {Gushel', N. P.},
  title = {On {Fano} varieties of genus 6},
  journal = {Mathematics of the USSR-Izvestiya},
  volume = {21},
  number = {3},
  pages = {445--459},
  year = {1983},
  doi = {10.1070/IM1983v021n03ABEH001801},
  url = {https://www.mathnet.ru/eng/im1701},
}

@article{debarre_kuznetsov_classification,
  author = {Debarre, Olivier and Kuznetsov, Alexander},
  title = {{Gushel}--{Mukai} varieties: classification and birationalities},
  journal = {Algebraic Geometry},
  volume = {5},
  number = {1},
  pages = {15--76},
  year = {2018},
  doi = {10.14231/AG-2018-002},
  eprint = {1510.05448},
  archivePrefix = {arXiv},
  primaryClass = {math.AG},
  url = {https://doi.org/10.14231/AG-2018-002},
}

@article{ogrady_epw,
  author = {O'Grady, Kieran G.},
  title = {Irreducible symplectic 4-folds and {Eisenbud}--{Popescu}--{Walter} sextics},
  journal = {Duke Mathematical Journal},
  volume = {134},
  number = {1},
  pages = {99--137},
  year = {2006},
  doi = {10.1215/S0012-7094-06-13413-0},
  eprint = {math/0507398},
  archivePrefix = {arXiv},
  primaryClass = {math.AG},
  url = {https://doi.org/10.1215/S0012-7094-06-13413-0},
}

@article{beauville_donagi,
  author = {Beauville, Arnaud and Donagi, Ron},
  title = {La vari\'{e}t\'{e} des droites d'une hypersurface cubique de dimension 4},
  journal = {Comptes Rendus de l'Acad\'{e}mie des Sciences. S\'{e}rie I. Math\'{e}matique},
  volume = {301},
  number = {14},
  pages = {703--706},
  year = {1985},
  url = {https://math.univ-cotedazur.fr/~beauvill/pubs/bd.pdf},
}

@article{auel_parimala_suresh,
  author = {Auel, Asher and Parimala, R. and Suresh, V.},
  title = {Quadric surface bundles over surfaces},
  journal = {Documenta Mathematica},
  volume = {Extra Volume},
  pages = {31--70},
  year = {2015},
  doi = {10.4171/DMS/7/3},
  eprint = {1207.4105},
  archivePrefix = {arXiv},
  primaryClass = {math.AG},
  url = {https://ems.press/books/dms/253/5002},
  note = {Alexander S. Merkurjev's Sixtieth Birthday},
}

@article{bernardara_bolognesi_conic_bundles,
  author = {Bernardara, Marcello and Bolognesi, Michele},
  title = {Derived categories and rationality of conic bundles},
  journal = {Compositio Mathematica},
  volume = {149},
  number = {11},
  pages = {1789--1817},
  year = {2013},
  doi = {10.1112/S0010437X13007392},
  eprint = {1010.2417},
  archivePrefix = {arXiv},
  primaryClass = {math.AG},
  url = {https://arxiv.org/abs/1010.2417},
}
\end{document}